\documentclass[notitlepage,11pt,reqno]{amsart}

\usepackage[a4paper,left=2cm,right=2.5cm,top=2.5cm, bottom = 2cm,headsep = 1cm ,letterpaper,margin=2.3cm]{geometry}
\usepackage{header}
\usepackage[utf8]{inputenc}
\usepackage[T1]{fontenc}
\usepackage{hyperref}
\hypersetup{
    colorlinks=true,
    allcolors=blue
}

\numberwithin{equation}{section}
\begin{document}

\title{Codimension-one foliations on adjoint varieties}

\author{Crislaine  \textsc{Kuster}} 

\address{\noindent Crislaine Kuster: IMPA, Estrada Dona Castorina 110, Rio de
  Janeiro, 22460-320, Brazil \newline
Université Bourgogne Europe, CNRS, IMB UMR 5584, F-21000 Dijon, France} 

\email{crislaine.kuster@impa.br // Crislaine.Kuster@u-bourgogne.fr // crislainekeizy@gmail.com}
\thanks{ }

\begin{abstract}
In this paper, we classify codimension-one foliations on adjoint varieties with most positive anti-canonical class. On adjoint varieties of Picard number one, we show that such foliations are always induced by a pencil of hyperplane sections with respect to the minimal embedding. For adjoint varieties of Picard number two, we prove that the space of such foliations contains more than one irreducible component, and we describe each of them. As a tool for understanding these foliations, we introduce the concept of the degree of a foliation with respect to a family of rational curves, which may be of independent interest.
\end{abstract}

\maketitle
\setcounter{tocdepth}{1}
{
\hypersetup{linkcolor=violet}

\renewcommand*\contentsname{Contents}

\setcounter{tocdepth}{1}

\renewcommand{\baselinestretch}{0}

\tableofcontents

\renewcommand{\baselinestretch}{1.0}\normalsize
  
}

\section{Introduction}
Let $X$ be a smooth complex projective variety of dimension $n$. A codimension-one foliation $\mathcal{F}$ on $X$ is defined by a rank-$(n-1)$ subsheaf $T_{\mathcal{F}}$ of the tangent bundle $T_X$ such that $T_{\mathcal{F}}$ is closed under the Lie bracket and saturated, i.e., the quotient $T_X/T_{\mathcal{F}}$ is torsion-free. The singular locus of $\mathcal{F}$, denoted by $\operatorname{Sing}(\mathcal{F})$, is the locus of points where $T_X/T_{\mathcal{F}}$ is not locally free. The sheaf $N_{\mathcal{F}} = (T_X/T_{\mathcal{F}})^{\vee\vee}$ is called the normal sheaf of $\mathcal{F}$; it is invertible. The canonical bundle of a foliation with normal bundle $\mathcal{N}$ can be defined as $K_{\mathcal{F}} = K_X \otimes \mathcal{N}$, where $K_X$ denotes the canonical bundle of $X$. Alternatively, $K_{\mathcal{F}}$ corresponds to the determinant of $T_{\mathcal{F}}^*$.

By the Frobenius integrability condition, the requirement that $T_{\mathcal{F}}$ is closed under the Lie bracket translates to the foliation $\mathcal{F}$ being defined by a $1$-form $[\omega] \in \mathbb{P}H^0(X, \Omega_X^1 \otimes N_{\mathcal{F}})$ such that
\[
\omega \wedge d\omega = 0.
\]
The condition that $T_{\mathcal{F}}$ is saturated translates to the assumption that $\operatorname{codim}_X \operatorname{Sing}(\omega) \ge 2$, where $\operatorname{Sing}(\omega) = \operatorname{Sing}(\mathcal{F})$ is the set of points in $X$ where the $1$-form $\omega$ vanishes. Consequently, if we fix a line bundle $\mathcal{N}$ on $X$, the space of codimension-one foliations on $X$ with normal bundle $\mathcal{N}$ is a quasi-projective variety given by
\[
\operatorname{Fol}(X, \mathcal{N}) = \bigl\{ [\omega] \in \mathbb{P}H^0(X, \Omega_X^1 \otimes \mathcal{N}) \;:\; \omega \wedge d\omega = 0,\; \operatorname{cod} \operatorname{Sing}(\omega) \ge 2 \bigr\} \subset \mathbb{P}H^0(X, \Omega_X^1 \otimes \mathcal{N}).
\]
The description of this space, for a fixed variety $X$ and a fixed line bundle $\mathcal{N}$, is an interesting problem in the global theory of holomorphic foliations. When $X = \mathbb{P}^n$, $n \ge 3$, a classical geometric invariant for codimension-one foliations is the degree, defined as the number of tangencies of the foliation with a general line. One can easily see that a codimension-one foliation of degree $d$ has normal sheaf $N_{\mathcal{F}} \cong \mathcal{O}_{\mathbb{P}^n}(d+2)$ and canonical bundle $K_{\mathcal{F}} \cong \mathcal{O}_{\mathbb{P}^n}(d-n+1)$.

The description of the space of codimension-one foliations on $\mathbb{P}^n$, $n \ge 3$, of a fixed degree $d$ is already a very challenging problem. A complete description of $\operatorname{Fol}(\mathbb{P}^n, \mathcal{O}_{\mathbb{P}^n}(d+2))$ is only known for $d \le 2$. It is a classical result that every codimension-one foliation with $d=0$ is given by a pencil of hyperplanes (see, e.g., \cite[Theorem 4.3]{ACM:FanoDist}). For $d=1$, there are two irreducible components in the space of codimension-one foliations, a result proved by Jouanolou in 1979 \cite{Jouanolou}. For $d=2$, the space $\operatorname{Fol}(\mathbb{P}^n, \mathcal{O}_{\mathbb{P}^n}(4))$ has six irreducible components, as proved by Cerveau and Lins Neto in 1996 \cite{CLN:components}. For $d=3$, there is a partial classification due to da Costa, Lizarbe, and Pereira in \cite{CLP:deg3}; they proved that there exist at least 24 components in the space of degree-three codimension-one foliations on $\mathbb{P}^n$.

Many researchers have worked on the problem of describing the moduli space of foliations for other varieties, imposing conditions either on the ambient variety or on the canonical bundle of the foliation. For example, in \cite{LPT13}, Loray, Pereira, and Touzet describe the space of foliations with trivial canonical bundle on Fano 3-folds. In \cite{AD-fano,AD17}, Araujo and Druel described the space of codimension-one foliations with ample anticanonical bundle on smooth projective varieties.

Another approach to this problem is to examine the behavior of foliations under restriction. Let $X$ be a smooth projective variety embedded in a projective space. One can investigate whether a codimension-one foliation on $X$ is a restriction of a foliation on the ambient space. In this direction, we refer to \cite{figueira} for quadrics, \cite{ACM:FanoDist} for complete intersections, and \cite{BFM} for cominuscule Grassmannians.

In \cite{BFM}, Benedetti, Faenzi, and Muniz focused on describing the space of codimension-one foliations on rational homogeneous varieties, i.e., varieties that admit a transitive action of a complex simple Lie group. Their work focuses on cominuscule Grassmannians, which are varieties whose tangent bundles are completely reducible. A natural next step is to study foliations on homogeneous varieties with tangent bundles that are not completely reducible. In this direction, our goal is to describe the space of codimension-one foliations on adjoint varieties.

Adjoint varieties are Fano varieties with a contact structure, which can be described as follows. Let $\mathfrak{g}$ be a simple Lie algebra and $G$ its Lie group. The group $G$ acts transitively on $\mathbb{P}\mathfrak{g}$ via the adjoint representation. The unique closed orbit of this action is called the \emph{adjoint variety} associated to $\mathfrak{g}$. We refer to the natural embedding $X \hookrightarrow \mathbb{P}\mathfrak{g}$ as the adjoint embedding ( see Section \ref{adjoint-var}). For the simple Lie algebra of type $A_n$, the adjoint variety has Picard number 2 and corresponds to a smooth hyperplane section of the product $\mathbb{P}^n \times \mathbb{P}^n$ under the Segre embedding. Except for this case, all remaining adjoint varieties have Picard number 1.

It is conjectured that any Fano contact variety is an adjoint variety; see, for instance, the survey \cite[Section 4.3]{beauville1999riemannian}. In a series of two papers \cite{kebekus2000uniqueness, Kebekus2003LinesOC}, Kebekus investigated the deformation theory of minimal rational curves on Fano contact varieties, showing that these varieties are covered by a proper family of rational curves. Motivated by this, to study codimension-one foliations on adjoint varieties, we begin with a more general setup: let $X$ be a smooth projective uniruled variety and $H$ a minimal dominating family of rational curves on $X$ (see Section \ref{ratcurves-ftang} for the definitions). We introduce the notion of the degree of a foliation with respect to a family of curves.

\begin{definition}\label{D:degree}
Let $X$ be a smooth projective uniruled variety and $H$ a dominating family of rational curves on $X$. The \emph{degree of a codimension-one foliation $\mathcal{F}$ on $X$ with respect to $H$}, denoted by $\deg_H \mathcal{F}$, is defined as the number of tangencies of $\mathcal{F}$ with a general curve in $H$ if this number is finite. If a general curve in $H$ is generically tangent to $\mathcal{F}$, we say that $\mathcal{F}$ has degree $-\infty$ with respect to $H$.
\end{definition}

Let $\mathcal{F}$ be a codimension-one foliation on $X$ and $H$ a dominating family of rational curves on $X$. Suppose that $\deg_H \mathcal{F} \ge 0$; then
\[
\deg_H \mathcal{F} = N_{\mathcal{F}} \cdot C - 2
\]
for a general curve $[C] \in H$. Note that when $X = \mathbb{P}^n$ and $H$ is the unique minimal dominating family of rational curves (namely, the family of lines), $\deg_H \mathcal{F}$ coincides with the classical definition of the degree of a codimension-one foliation on $\mathbb{P}^n$. If $X$ has Picard number one, there is no foliation tangent to the general curve in $H$ (see Proposition \ref{P:deg-maior-que-0}). In other words, the degree of a codimension-one foliation is always non-negative in that case.

Now let $X \subset \mathbb{P}^N$ be a smooth subvariety covered by lines, and let $H$ be a minimal dominating family of lines on $X$. If $\mathcal{F}$ is a codimension-one foliation on $X$ induced by a degree-zero codimension-one foliation on $\mathbb{P}^N$, then $\deg_H \mathcal{F} = 0$. This leads to the following question: when does the converse hold?

Recall that a smooth projective uniruled variety with Picard number one is Fano (see Remark~\ref{Fano}). Thus, to address this question, we first study codimension-one foliations on smooth Fano varieties of Picard number one with degree zero with respect to a minimal dominating family of rational curves. In particular, this includes adjoint varieties of Picard number one. We obtain the following result.

\begin{theorem}\label{mainA}
Let $\mathcal{F}$ be a codimension-one foliation on a smooth Fano variety $X$ of Picard number one, and let $H$ be a minimal dominating family of rational curves on $X$. Suppose that $\deg_H \mathcal{F} = 0$. Then $\mathcal{F}$ is algebraically integrable; that is, $\mathcal{F}$ is induced by a dominant rational map $X \dashrightarrow \mathbb{P}^1$.
\end{theorem}The adjoint variety associated with the Lie algebra \( \mathfrak{sp}(2n) \) is the projective space \( \mathbb{P}^{2n-1} \) under the second Veronese embedding  
\[
\nu_2\colon \mathbb{P}^{2n-1} \hookrightarrow \mathbb{P}(\mathfrak{sp}(2n)),
\]  
which is not the minimal embedding. Note that the Veronese variety \(\nu_2(\mathbb{P}^{2n-1})\) does not contain lines from the ambient space; rather, it is covered by conics.
The pullback of a codimension-one foliation of degree zero from the ambient space via \( \nu_2 \) is a codimension-one foliation of degree two on \( \mathbb{P}^{2n-1} \).  

For \(n\ge 3\), the space of codimension-one foliations of degree two on \( \mathbb{P}^{2n-1} \) has six irreducible components, whereas the space of codimension-one foliations of degree zero on \( \mathbb{P}(\mathfrak{sp}(2n)) \) has only one irreducible component. Consequently, not every foliation on the adjoint variety \( X \) associated with \( \mathfrak{sp}(2n) \) is the restriction, under the embedding \( X \hookrightarrow \mathbb{P}(\mathfrak{sp}(2n)) \), of a foliation from the ambient space.

If \( X \) is an adjoint variety with Picard number one that is not isomorphic to \( \mathbb{P}^{2n-1} \), then \( X \) is covered by lines under the adjoint embedding, which in this case coincides with the minimal embedding (see \cite[Section 2.3]{kebekus2000uniqueness}). In this setting, we have the following result:

\begin{theorem}\label{T: adjoind picard one}
Let \( X \) be an adjoint variety with Picard number one not isomorphic to \( \mathbb{P}^n \), and let \(H\) be a minimal dominating family of rational curves on \(X\). Let \(\mathcal{F}\) be a codimension-one foliation on \(X\) with \(\deg_H \mathcal{F}=0\). Then the foliation \(\mathcal{F}\) is a pencil of hyperplane sections under the adjoint embedding.
\end{theorem}

Therefore, we conclude that the space of codimension-one foliations of degree zero with respect to a minimal dominating family of rational curves on adjoint varieties of Picard number one has only one irreducible component. 

Furthermore, we consider codimension-one foliations on adjoint varieties of Picard number two, that is, general hyperplane sections of \(\mathbb{P}^n \times \mathbb{P}^n\) under the Segre embedding, \(n\ge 2\). Consider the two natural projections \(\pi_1 \colon X \longrightarrow \mathbb{P}^n\) and \(\pi_2 \colon X \longrightarrow \mathbb{P}^n\).
For \(i = 1,2\), let \(h_{i}\) be the pullback by \(\pi_i\) of the hyperplane class on \(\mathbb{P}^n\). Then \(\mathcal{O}_X(1) = \mathcal{O}_X(h_{1} + h_{2})\) and \(\omega_X = \mathcal{O}_X(-n)\). 
There are two minimal dominating families of rational curves on \(X\): one consists of lines contained in the fibers of the first projection \(\pi_1\), denoted by \(H_1\), and the other, denoted by \(H_2\), consists of lines contained in the fibers of the second projection \(\pi_2\).

We aim to study codimension-one foliations of minimal degrees with respect to these two families. Our first result in this direction is the following:

\begin{proposition}\label{P:Picard two}
Let \( X \) be a smooth hyperplane section of \(\mathbb{P}^n \times \mathbb{P}^n\) and let \( \mathcal{F} \) be a codimension-one foliation on \(X\). If \( \deg_{H_1} \mathcal{F} = -\infty \), then \( \mathcal{F} \) is the pullback by \(\pi_1\) of a codimension-one foliation on \( \mathbb{P}^n\) of degree equal to \( \deg_{H_2} \mathcal{F}\).
\end{proposition}

Consider now the adjoint embedding of \( X \) into a projective space, \( s\colon X \hookrightarrow \mathbb{P}^N \). Observe that the pullback of a codimension-one foliation of degree zero on \( \mathbb{P}^N \) via \( s \) defines a codimension-one foliation on \( X \) with degree zero with respect to both families. This raises the following question: are all foliations on \( X \) with degree zero with respect to both families obtained as the pullback of a degree zero foliation on \( \mathbb{P}^N \)?  

Since the space of codimension-one foliations of degree zero on \( \mathbb{P}^N \) has a unique irreducible component, the answer to this question is no, as shown in the following result:

\begin{theorem}\label{T:Picard two}
Let \( X \) be a smooth hyperplane section of \(\mathbb{P}^n \times \mathbb{P}^n\) and let \(H_1\) and \(H_2\) be the families of lines contained in fibers of the two natural projections of \(X\) onto \(\mathbb{P}^n\). Let \( \mathcal{F} \) be a codimension-one foliation on \(X\). 
Suppose that 
\[
\deg_{H_i} \mathcal{F}=0,\quad \text{for } i=1,2.
\]
If \(n \ge 3\), then there are two irreducible components in the space of codimension-one foliations with degree zero with respect to both families (see Proposition \ref{P: logarithmic on PnxPn}). 
If \(n=2\), there also exists a rigid foliation associated with the affine Lie algebra, resulting in three irreducible components in the space of codimension-one foliations of degree zero with respect to both families (see Theorem \ref{teo-comp-A2}).
\end{theorem}

Throughout this work, we always assume the base field is \( \mathbb{C} \). 

The paper is organized as follows. 
In Section \ref{sec-fol}, we introduce fundamental definitions and results on holomorphic foliations, slope stability for foliations, and a key tool for our work: the tangential foliation. 
In Section \ref{S:deg fol ratcurves}, we define the degree of a foliation with respect to a family of rational curves. 
Section \ref{S: Rat homog var} is devoted to rational homogeneous varieties, with particular emphasis on adjoint varieties. 
Section \ref{sec4-Fano-picard1} studies foliations on Fano varieties with Picard number one, with special attention to adjoint varieties of Picard number one (Section \ref{sec4-picard1}). 
Finally, Section \ref{sec5-Picard2} examines foliations on adjoint varieties with Picard number two.
Accordingly, in Section \ref{sec4-Fano-picard1}, we prove Theorem \ref{mainA} using techniques involving deformations of rational curves along foliations (see Proposition \ref{main}). More generally, we establish a bound on the algebraic rank of a foliation in terms of the algebraic rank of its tangential foliation (see Proposition \ref{P:bound algebraic rank}).
We then prove Theorem \ref{T: adjoind picard one}, with key steps presented in Propositions \ref{P: pencil} and \ref{P: adjoint H^0}.
In Section \ref{sec5-Picard2}, Proposition \ref{P:Picard two} and Theorem \ref{T:Picard two} are proved; see Proposition \ref{deg infinito}, Theorem  \ref{P: logarithmic on PnxPn}, and Section \ref{S: adjoint sl3}.

\section{Foliations on algebraic varieties }\label{sec-fol}

\subsection{Basic concepts}\label{subsec-fol-basicconcepts}
Let $X$ be a smooth complex projective variety.

\begin{definition}\label{Def-fol}
A \emph{foliation} $\mathcal{F}$ on $X$ is defined by a coherent subsheaf $T_{\mathcal{F}}$ of the tangent sheaf $T_X$, called the tangent sheaf of $\mathcal{F}$, such that
\begin{enumerate}
    \item $T_{\mathcal{F}}$ is involutive, that is, it is closed under the Lie bracket, and
    \item $T_{\mathcal{F}}$ is saturated, that is, the quotient $T_X/T_{\mathcal{F}}$ is torsion-free.
\end{enumerate}
More generally, a \emph{distribution} $\mathcal{D}$ on $X$ is defined by a saturated coherent subsheaf $T_{\mathcal{D}} \subset T_X$; that is, it satisfies only the second condition above.
\end{definition}

The locus of points where $T_X/T_{\mathcal{F}}$ is not locally free is called the \emph{singular locus} of $\mathcal{F}$ and is denoted by $\operatorname{Sing}(\mathcal{F})$. Since $T_X/T_{\mathcal{F}}$ is torsion-free, $\operatorname{codim} \operatorname{Sing}(\mathcal{F}) \ge 2$. The \emph{rank} of $\mathcal{F}$, denoted by $r_{\mathcal{F}}$, is the generic rank of $T_{\mathcal{F}}$, and the \emph{codimension} of $\mathcal{F}$ is $n - r_{\mathcal{F}}$.  
The \emph{normal sheaf} of the foliation is the sheaf $\mathcal{N}_{\mathcal{F}} = (T_X/T_{\mathcal{F}})^{\vee \vee}$. The \emph{canonical class} $K_{\mathcal{F}}$ of $\mathcal{F}$ is any divisor on $X$ such that  
\[
\mathcal{O}_X(-K_{\mathcal{F}}) \simeq \det(T_{\mathcal{F}}).
\]  
When $-K_{\mathcal{F}}$ is ample, we define the \emph{index} $\iota_{\mathcal{F}}$ of the foliation $\mathcal{F}$ on $X$ to be the largest integer dividing $-K_{\mathcal{F}}$ in $\operatorname{Pic}(X)$.
Over the smooth locus $X^0 = X \setminus \operatorname{Sing}(\mathcal{F})$, we have the following exact sequence
\[
0 \longrightarrow T_{\mathcal{F}}|_{X^0} \longrightarrow T_{X^0} \longrightarrow \mathcal{N}_{\mathcal{F}}|_{X^0} \longrightarrow 0.
\]

\begin{definition}
Let $\mathcal{F}$ be a foliation on $X$.
An analytic submanifold $L \subset X$ of dimension equal to $r_{\mathcal{F}}$ is called a \emph{leaf} of the foliation $\mathcal{F}$ if the map $T_{\mathcal{F}}|_L \longrightarrow T_X|_L$ factors through $T_L$.
A closed subvariety $Y \subset X$ is said to be \emph{invariant} by $\mathcal{F}$ if it is not contained in the singular locus of $\mathcal{F}$ and the map $T_{\mathcal{F}}|_Y \longrightarrow T_X|_Y$ factors through $T_Y$.
\end{definition}

Let $\mathcal{F}$ be a codimension-$q$ foliation on a smooth projective variety $X$. Such a foliation can be described by a $q$-form $\omega$ with coefficients in the line bundle $\det(\mathcal{N}_{\mathcal{F}}) = \big(\bigwedge^q \mathcal{N}_{\mathcal{F}}\big)^{\vee \vee}$. Specifically, the foliation $\mathcal{F}$ is defined by a global section 
\[
\omega \in H^0\big(X, \Omega_X^q \otimes \det(\mathcal{N}_{\mathcal{F}})\big),
\] 
where, by the Frobenius Theorem, this $p$-form satisfies  the following properties:
\begin{enumerate}
    \item[(a)] \textit{Local Decomposability}: At a general point of $X$, the germ of $\omega$ can be expressed as a wedge product of $q$ holomorphic $1$-forms:
    \[
    \omega = \omega_1 \wedge \cdots \wedge \omega_q.
    \]
    
    \item[(b)] \textit{Integrability}: For a general point of $X$, the decomposition of $\omega$ from (a) satisfies the integrability condition:
    \[
    d\omega_i \wedge \omega = 0 \quad \text{for every } i = 1, \ldots, q.
    \]
\end{enumerate}

In the next definition, we introduce a global invariant of a foliation on a smooth projective variety, namely the \emph{algebraic rank}; see \cite[Definition~2.4]{AD-19} and \cite[Definition~2]{AD17}.

\begin{definition}
Let $\mathcal{F}$ be a foliation on a smooth projective variety $X$. The \emph{algebraic rank} $r_a(\mathcal{F})$ of $\mathcal{F}$ is the maximum dimension of an algebraic subvariety $Z$ passing through a general point of $X$ and tangent to $\mathcal{F}$. That is, for every point $z \in Z \setminus \operatorname{Sing}(\mathcal{F})$, we have $T_z Z \subset T_z \mathcal{F}$.
We say that a foliation $\mathcal{F}$ is \emph{purely transcendental} if its algebraic rank is zero. 
In contrast, a foliation $\mathcal{F}$ on $X$ is said to be \emph{algebraically integrable} if its algebraic rank equals its dimension; that is, the leaf of $\mathcal{F}$ through a general point is an algebraic variety. Equivalently, $\mathcal{F}$ is induced by a dominant rational map with irreducible general fiber, $X \dashrightarrow Y$, where $Y$ is a smooth projective variety (see \cite[Section 2.3]{loray-pereira-touzet:trivial}). When $\mathcal{F}$ has codimension one, such a map is called a \emph{first integral} of $\mathcal{F}$.
\end{definition}

\begin{definition}
Let $\mathcal{F}$ be a foliation on a smooth projective variety $X$, and let $\pi: Y \dashrightarrow X$ be a dominant rational map from a smooth variety $Y$. We say that $\mathcal{G}$ is the pullback of $\mathcal{F}$ via $\pi$ if for smooth open subsets $X^\circ \subset X$ and $Y^\circ \subset Y$ such that $\pi$ restricts to a smooth morphism $\pi^\circ: Y^\circ \longrightarrow X^\circ$, we have  
\[
\mathcal{G}|_{Y^\circ} = (d\pi^\circ)^{-1} (\mathcal{F}|_{X^\circ}).
\]  
In this case, we write $\mathcal{G} = \pi^{*} \mathcal{F}$.
\end{definition}

\subsection{The space of codimension-one foliations on projective varieties}\label{S:components}
Since our work focuses on moduli spaces of codimension-one foliations, we will assume from now on that $\mathcal{F}$ is a codimension-one foliation on a smooth projective variety $X$.
Fix a line bundle $\mathcal{N} \in \operatorname{Pic}(X)$. A codimension-one foliation $\mathcal{F}$ with normal sheaf isomorphic to $\mathcal{N}$ is defined by a $1$-form $\omega \in H^0(X, \Omega^1_X \otimes \mathcal{N})$ satisfying
\[
\omega \wedge d\omega = 0 \quad \text{and} \quad \operatorname{cod} \operatorname{Sing}(\omega) \ge 2,
\]
where $\operatorname{Sing}(\omega) = \operatorname{Sing}(\mathcal{F})$ denotes the set of points in $X$ where the $1$-form $\omega$ vanishes.
We define the space of codimension-one foliations on $X$ with normal bundle $\mathcal{N}$ to be the quasi-projective variety
\[
\operatorname{Fol}(X, \mathcal{N}) = \left\{ [\omega] \in \mathbb{P}H^0(X, \Omega^1_X \otimes \mathcal{N}) \;\middle|\; \omega \wedge d\omega = 0,\; \operatorname{cod} \operatorname{Sing}(\omega) \ge 2 \right\}
\subset \mathbb{P}H^0(X, \Omega^1_X \otimes \mathcal{N}).
\]

Some classical irreducible components of the space of codimension-one foliations include logarithmic foliations, linear pullbacks, and those associated to actions of the $2$-dimensional affine Lie algebra, or more generally, foliations tangent to algebraic actions (see \cite[Sections 4 and 5]{CLP:deg3}). In what follows, we briefly describe each of these components.

\subsubsection{Logarithmic components}\label{S:logarithmic components}
Let $X$ be a simply connected smooth projective variety of dimension $n \ge 3$. Consider a closed rational $1$-form $\omega$ on $X$ without codimension-one zeros. Suppose $\omega$ has simple poles along the support of a simple normal crossing divisor $D := \sum D_i$, where
\[
D = -\operatorname{div}(\omega) = (\omega)_\infty - (\omega)_0,
\]
and each $D_i$ is an irreducible hypersurface. Since $\omega$ is closed, it induces a foliation $\mathcal{F}$ with normal sheaf $\mathcal{O}_X(D)$, defined as the kernel of the map
\[
T_X \longrightarrow \mathcal{O}_X(D)
\]
given by contraction with $\omega$.

Now, let $\lambda_1, \dots, \lambda_k \in \mathbb{C}$ satisfy $\sum \lambda_i c_1(D_i) = 0$. Then there exists a unique closed rational $1$-form $\omega$ with residues along each $D_i$ equal to $\lambda_i$ and such that
\[
D = \sum_{i=1}^{k} D_i = -\operatorname{div}(\omega) = (\omega)_\infty - (\omega)_0,
\]
(see \cite[Theorem 3.1]{Pereira22}).  
Recall that the residue of a $1$-form along a divisor is defined as follows. Let $x \in D_i$ be a smooth point and let $p\colon \overline{\mathbb{D}} \hookrightarrow X$ be an embedding of a closed disc, holomorphic in the interior, intersecting $D_i$ transversely at $x$. Then
\[
\operatorname{Res}_{D_i}(\omega) = \frac{1}{2\pi i} \int_{\partial \mathbb{D}} p^* \omega \in \mathbb{C}.
\]
The closedness of $\omega$ and the irreducibility of $D_i$ ensure that this integral is independent of the choices of $x$ and $p$.
 Consider $\operatorname{Log}(D)$ to be the closure in
\(
\mathbb{P} H^0\big(X, \Omega^1_X \otimes \mathcal{O}_X(D)\big)
\)
of the set of closed $1$-forms $\omega$ satisfying
\(
\sum_{i=1}^{k} D_i = -\operatorname{div}(\omega) = (\omega)_\infty - (\omega)_0.
\)
The elements of $\operatorname{Log}(D)$ are called \emph{logarithmic $1$-forms}. Locally, they can be written as
\[
\omega = \alpha + \sum_{i=1}^{k} g_i \frac{df_i}{f_i},
\]
where $\{f_1 \cdots f_k = 0\}$ is a reduced local equation for the support of $D$, $\alpha$ is holomorphic, and $g_1, \dots, g_k$ are holomorphic functions. 
Under certain assumptions, Calvo-Andrade proved that $\operatorname{Log}(D)$ is an irreducible component of $\operatorname{Fol}(X, N_{\mathcal{F}})$ (see \cite{CalvoAndrade1994IrreducibleCO}). Later, Loray, Pereira, and Touzet extended this result to smooth projective varieties satisfying $H^0(X, \Omega^1_X) = 0$ (see \cite[Lemma 2.5]{LPT13}). These components are called \emph{logarithmic components}.

If $X$ has Picard number one, the construction of these components is rather simple (see \cite[Section 2.5]{LPT13}). Let $D = \sum_{i=1}^k D_i$ be a simple normal crossing divisor. 
Suppose that the divisors $D_i$ have degrees $d_i$ and $\sum_{i=1}^k d_i = d$. We define the rational map  
\[
\Phi\colon \Sigma \times \left( \prod_{i=1}^{k} \mathbb{P} H^0(X, \mathcal{O}_X(d_i)) \right)
\longrightarrow \mathbb{P} H^0 \left( X, \Omega^1_X(d) \right)
\]
\[
\bigl( (\lambda_1 : \cdots : \lambda_k), f_1, \dots, f_k \bigr) \mapsto 
\left( \prod_{i=1}^{k} f_i \right) \left( \sum_{i=1}^{k} \lambda_i \frac{df_i}{f_i} \right),
\]
where $\Sigma \subset \mathbb{P}^{k-1}$ is the hyperplane given by
\(
\left\{ \sum \lambda_i d_i = 0 \right\},
\)
where this equality corresponds to the condition on the Chern class of the residue divisor (see \cite[Proposition 3.2]{CLP:deg3}).
The closure of the image of $\Phi$ is an irreducible component of $\operatorname{Fol}(X, \mathcal{O}_X(d-2))$, denoted by $\operatorname{Log}(d_1, \dots, d_k)$. When $k = 2$, the foliations on these components are algebraically integrable and they are also known as \emph{rational components}.

\subsubsection{Components defined by linear pullbacks}
Let $\mathcal{F}$ be a codimension-one foliation of degree $k$ on $\mathbb{P}^2$, and let 
$\pi\colon \mathbb{P}^n \dashrightarrow \mathbb{P}^2$ be a linear projection with $n \ge 3$. 
Then the pullback foliation $\mathcal{F}^* := \pi^*\mathcal{F}$ is a codimension-one foliation of degree $k$ on $\mathbb{P}^n$.
The space of foliations on $\mathbb{P}^n$ arising as linear pullbacks of degree-$k$ foliations on $\mathbb{P}^2$ forms an irreducible component, denoted by $\operatorname{LPB}(n,k)$.
Such foliations were first studied by Camacho and Lins Neto \cite{Camacho1982}, and later Cerveau, Lins Neto, and Edixhoven \cite{CLNE-pulback} showed that they indeed give rise to irreducible components in the space of foliations on $\mathbb{P}^n$.

\subsubsection{Foliations associated to algebraic actions}\label{E: fol group actions}
Let \( X \) be a smooth projective variety, and let \( G \subset \operatorname{Aut}(X) \) be a connected Lie group with Lie algebra \( \mathfrak{g} \). Assume that the natural action of \( G \) on \( X \) satisfies that \( \dim (G \cdot x) \) is constant outside a subset of codimension at least two. It follows that the normal sheaf of the foliation induced by this action, \( N_{\mathcal{F}} \), is torsion-free. Denote by \( \mathcal{F} \) the foliation induced by this action on \( X \), namely, the foliation whose leaves are the orbits of maximal dimension.  
The differential of the morphism $G \to \operatorname{Aut}(X)$ yields a map
\(
\mathfrak{g} \longrightarrow H^0(X, T_X).
\)
Tensoring with the structure sheaf $\mathcal{O}_X$, we obtain a morphism
\[
\rho : \mathfrak{g} \otimes_{\mathbb{C}} \mathcal{O}_X \longrightarrow T_X.
\]
Then the tangent sheaf of $\mathcal{F}$ is $T_{\mathcal{F}} \cong \operatorname{Im}(\rho)$.

For a fixed variety $X$, one may be interested in determining when a foliation on $X$ is associated with the action of a connected Lie group, and in such cases, when the induced foliation is \emph{rigid}. Here, \emph{rigid} means that the foliation does not admit nontrivial deformations; that is, any deformation of $\mathcal{F}$ is induced by an automorphism of $X$. Note that this construction produces foliations of potentially higher codimension. For a broader perspective on these foliations, see \cite{velazquez2024moduli}.

For codimension-one foliations, we have a positive answer to the above problem in the following cases. When $X = \mathbb{P}^3$, there exists a codimension-one rigid foliation $\mathcal{F}$ of degree two induced by an action of the $2$-dimensional affine Lie algebra $\mathfrak{aff}(\mathbb{C}) \subset \mathfrak{aut}(\mathbb{P}^3)$. Cerveau and Lins Neto \ first described this component, see cite{CLN:components}, and they  also showed that its linear pullback to $\mathbb{P}^n$ defines a component of 
$\operatorname{Fol}(\mathbb{P}^n, \mathcal{O}_{\mathbb{P}^n}(4))$, $n \ge 3$.  
Furthermore, Loray, Pereira, and Touzet \cite[Section 5]{LPT13} studied a rigid foliation induced by the $2$-dimensional affine Lie algebra $\mathfrak{aff}(\mathbb{C})$ on the three-dimensional quadric. In Example \ref{example-aff}, we also describe a rigid foliation arising from the action of $\mathfrak{aff}(\mathbb{C})$ on a smooth hyperplane section of $\mathbb{P}^2 \times \mathbb{P}^2$ under the Segre embedding.

\subsection{Stability of foliations}
Let $X$ be an $n$-dimensional projective variety, and let $\mathcal{A}$ be an ample line bundle on $X$. Let $\mathcal{E}$ be a torsion-free sheaf of rank $r>0$ on $X$. We define the slope of $\mathcal{E}$ with respect to $\mathcal{A}$ as
\[
\mu_{\mathcal{A}}(\mathcal{E}) = \frac{c_1(\mathcal{E}) \cdot \mathcal{A}^{n-1}}{r}.
\]
We say that $\mathcal{E}$ is
\emph{$\mu_{\mathcal{A}}$-stable} if for any nonzero
subsheaf $\mathcal{G} \subsetneq \mathcal{E}$ we have $\mu_{\mathcal{A}}(\mathcal{G}) < \mu_{\mathcal{A}}(\mathcal{E})$, and
\emph{$\mu_{\mathcal{A}}$-semistable} if we have $\mu_{\mathcal{A}}(\mathcal{G}) \le \mu_{\mathcal{A}}(\mathcal{E})$.
Given a torsion-free sheaf $\mathcal{E}$ on $X$, there exists a filtration of $\mathcal{E}$ by subsheaves:
\[
0 = \mathcal{E}_0 \subsetneq \mathcal{E}_1 \subsetneq \cdots \subsetneq \mathcal{E}_k = \mathcal{E},
\]
with $\mu_{\mathcal{A}}$-semistable quotients $\mathcal{Q}_i = \mathcal{E}_i / \mathcal{E}_{i-1}$, such that:
\[
\mu_{\mathcal{A}}(\mathcal{Q}_1) > \mu_{\mathcal{A}}(\mathcal{Q}_2) > \cdots > \mu_{\mathcal{A}}(\mathcal{Q}_k).
\]
This filtration is called the \textit{Harder-Narasimhan filtration} of $\mathcal{E}$. The sheaf $\mathcal{E}_1$, known as the \emph{maximal destabilizing subsheaf} of $\mathcal{E}$, is uniquely determined, semistable, and satisfies the following condition: for all subsheaves $\mathcal{G} \subset \mathcal{E}$, we have  
\[
\mu_{\mathcal{A}}(\mathcal{E}_1) \ge \mu_{\mathcal{A}}(\mathcal{G}),
\]
and in case of equality, $\mathcal{E}_1 \supset \mathcal{G}$ (see for instance \cite[Definition 1.3.6]{Huybrechts_Lehn_2010}). 
When $\mu_{\mathcal{A}}(T_X) > 0$, the sheaf $\mathcal{E}_1 \subsetneq T_X$ defines a foliation on $X$ as shown in the following:

\begin{lemma}[{\cite[Lemma 7.2]{AD-fano}}] \label{Lemma 2.5}
Let $X$ be a smooth projective variety, $\mathcal{A}$ an ample line bundle on $X$, and $\mathcal{F} \subseteq T_X$ a foliation on $X$. Let $0 = {\mathcal{F}_0} \subset {\mathcal{F}_1} \subset \cdots \subset {\mathcal{F}_k} = {\mathcal{F}}$ be the Harder-Narasimhan filtration of ${\mathcal{F}}$ with respect to $\mathcal{A}$. Then ${\mathcal{F}_i} \subseteq T_X$ defines a foliation on $X$ for every $i$ such that $\mu_{\mathcal{A}}(T_{\mathcal{F}_i}) > 0$.
\end{lemma}

We say that a foliation $\mathcal{F}$ is $\mu_{\mathcal{A}}$-stable (or $\mu_{\mathcal{A}}$-semistable) if its tangent sheaf $T_{\mathcal{F}}$ is $\mu_{\mathcal{A}}$-stable (or $\mu_{\mathcal{A}}$-semistable, respectively). As a consequence of Lemma \ref{Lemma 2.5}, we have the following  
\begin{corollary}[{\cite[Proposition 2.2]{LPT13}}]
Let $\mathcal{F}$ be a foliation on a polarized smooth projective variety $(X, \mathcal{A})$ satisfying $\mu_{\mathcal{A}}(T_{\mathcal{F}}) \ge 0$. If $T_{\mathcal{F}}$ is not semistable, then the maximal destabilizing subsheaf $\mathcal{G}$ of $T_{\mathcal{F}}$ is involutive. Thus, $\mathcal{G}$ is a semistable foliation on $X$ tangent to $\mathcal{F}$ and satisfying
\[
\mu_{\mathcal{A}}(T_{\mathcal{G}}) > \mu_{\mathcal{A}}(T_{\mathcal{F}}).
\]
\end{corollary}

\subsection{Transverse structures for codimension-one foliations}\label{S: tansversality of fol}
In this section we introduce the concept of transversality of foliations; for more details, the reader can check \cite[Section 2]{LPT-def} and the references therein. Such structures form a chaining for foliations, descending to algebraically integrable foliations. Let $X$ be a smooth complex projective variety. Recall that a Riccati foliation $\mathcal{H}$ on a $\mathbb{P}^1$-bundle $\pi \colon P \longrightarrow X$ is a foliation on $P$ that has no tangencies with the general fiber of $\pi$.

Let $\mathcal{F}$ be a codimension-one foliation on $X$ and consider $\omega$ a rational $1$-form defining $\mathcal{F}$.
We say that $\mathcal{F}$ admits a \textit{transversely projective structure} if there exist a $\mathbb{P}^1$-bundle $\pi \colon P \longrightarrow X$, a Riccati foliation $\mathcal{H}$ on $P$, and a meromorphic section $\sigma \colon X \dashrightarrow P$ such that $\sigma$ is generically transverse to $\mathcal{H}$ and the restriction of $\mathcal{H}$ to $\sigma(X)$ corresponds to $\mathcal{F}$ via $\sigma$. We then say that $(P, \mathcal{H}, \sigma)$ is a transversely projective structure for $\mathcal{F}$.  

Two such triples $\mathcal{P} = (P, \mathcal{H}, \sigma)$ and $\mathcal{P}' = (P', \mathcal{H}', \sigma')$ are said to be equivalent when they are conjugate by a birational bundle transformation $\phi \colon P \dashrightarrow P'$ satisfying $\phi^* \mathcal{H}' = \mathcal{H}$ and $\phi \circ \sigma = \sigma'$. According to \cite{CLLPT04}, the existence of a transversely projective structure is equivalent to the existence of nonzero rational $1$-forms $\omega_1$ and $\omega_2$ on $X$ such that
\[
\begin{cases}
    d\omega = \omega \wedge \omega_1, \\
    d\omega_1 = \omega \wedge \omega_2, \\
    d\omega_2 = \omega_1 \wedge \omega_2.
\end{cases}
\]
In the next example, we show that a Riccati foliation is transversely projective.

\begin{example}
Let $\mathcal{F}$ be a Riccati foliation on $P$, a $\mathbb{P}^1$-bundle over $X$. Then $\mathcal{F}$ is defined, in a trivialization $\mathcal{U} \times \mathbb{P}^1$ for $\mathcal{U}$ an affine open subset of $X$, by an integrable $1$-form  
\[
\Omega = dz + \omega_0 + z\omega_1 + \frac{z^2}{2} \omega_2,
\]  
where $\omega_0, \omega_1, \omega_2$ are pullbacks of rational forms on $\mathcal{U}$.
Consider the vector field $v = \frac{\partial}{\partial z}$. Then $v$ satisfies $\Omega(v) = 1$, and the third Lie derivative $\mathcal{L}_v^{(3)}\Omega$ of $\Omega$ along $v$ is zero. Thus, the foliation $\mathcal{F}$ is transversely projective; see \cite{CACGL04}. For an alternative proof, see also \cite[Section 3]{loray2007transversely}.
Note that the integrability of $\Omega$ implies the integrability of $\omega_0$ and ensures that $(\omega_0, \omega_1, \omega_2)$ defines a transversely projective structure for the foliation on $\mathcal{U}$ defined by $\omega_0$.
\end{example}
 
We say that $\mathcal{F}$ admits a \textit{transversely affine structure} if it admits a transversely projective structure $(P, \mathcal{H}, \sigma)$ as above and the section $\sigma \colon X \longrightarrow P$ is invariant by $\mathcal{H}$. Equivalently, $\mathcal{F}$ is transversely affine if there exists a nonzero rational $1$-form $\omega_1$ on $X$ such that
\[
\begin{cases}
    d\omega = \omega \wedge \omega_1, \\
    d\omega_1 = 0.
\end{cases}
\]
In general, a Riccati foliation  $\mathcal{F}$  on a $\mathbb{P}^1$-bundle $\pi\colon P \longrightarrow X$ over $X$ does not necessarily admit a transversely affine structure (see \cite[Example 2.18]{CLLPT04}). It does so if and only if there exists a hypersurface $H \subset P$ that is invariant under $\mathcal{F}$ and dominates $X$ (i.e., $\pi(H) = X$); see \cite[Example 2.11]{CP14} and \cite[Remark 2.17]{CLLPT04}.

If a foliation $\mathcal{F}$ admits a transversely affine structure such that the associated $1$-form $\omega_1$ is logarithmic and all its periods are commensurable with $2\pi\sqrt{-1}$ (i.e., they belong to $2\pi\sqrt{-1} \mathbb{Q}$), then $\mathcal{F}$ is said to be \emph{virtually transversely Euclidean}. Finally, we say that $\mathcal{F}$ admits a \textit{transversely Euclidean structure} if $\omega_1$ is logarithmic with all its periods being integral multiples of $2\pi\sqrt{-1}$. In this latter case, such foliations are defined by closed rational $1$-forms. 
In \cite[Example 2.8]{CP14} and \cite[Example 3.2]{LPT-def}, the authors provide examples of foliations on $\mathbb{P}^1 \times \mathbb{P}^1$ that are transversely affine but not defined by a closed rational $1$-form; see also \cite[Theorem D]{LPT16-representations} and \cite[Theorem A]{CP14}.

These transversal structures form a type of hierarchy for the algebraic integrability of the foliation $\mathcal{F}$. More precisely, transversely projective foliations admit a canonical collection of local holomorphic first integrals defined on the complement of the polar set of $\omega_0, \omega_1, \omega_2$. A transversely affine structure for $\mathcal{F}$ determines a local system of first integrals given by the branches of the multi-valued function
\[
F = \int \left( \exp \int \omega_1 \right) \omega_0.
\]
Finally, if the foliation $\mathcal{F}$ admits a transversely Euclidean structure, then $\exp \int \omega_1$ is a rational function and $\left( \exp \int \omega_1 \right) \omega_0$ is a closed rational $1$-form defining $\mathcal{F}$.

These structures behave nicely with respect to dominant rational maps, as illustrated by the following lemma.

\begin{lemma}[{\cite{Casale02}, Lemma 2.1 and Lemma 3.1}]\label{L: pullback tansv}
Let $X$ and $Y$ be smooth projective varieties, $\pi\colon X \dashrightarrow Y$ a dominant rational map, and $\mathcal{F}$ a codimension-one foliation on $Y$. The foliation $\pi^* \mathcal{F}$ is algebraically integrable, transversely projective, transversely affine, or virtually transversely Euclidean if and only if $\mathcal{F}$ is.
\end{lemma}

Note that by pulling back the $1$-forms $\omega_i$ defining the corresponding structure via $\pi$, we conclude that if $\mathcal{F}$ is transversely projective, transversely affine, virtually transversely Euclidean, or transversely Euclidean, then the same holds for $\mathcal{G}$.
Although we do not have a version of Lemma \ref{L: pullback tansv} for transversely Euclidean foliations in general (see \cite[Example 1.5.5]{Kuster2025}), in certain contexts and under additional hypotheses we will see that such a result holds for transversely Euclidean foliations (see Lemma \ref{L:pullback transv euclidean}).  

The existence of multiple non-equivalent transverse structures on a foliation $\mathcal{F}$ imposes strong constraints on $\mathcal{F}$, as shown in the lemma below.

\begin{lemma}[{\cite[Lemma 2.2]{LPT-def}}]\label{L:2transv aff structures}
Let $\mathcal{F}$ be a codimension-one foliation on a smooth projective variety. The following assertions hold true.
\begin{enumerate}
    \item If $\mathcal{F}$ admits two non-equivalent transversely projective structures, then $\mathcal{F}$ admits a virtually transversely Euclidean structure.
    \item If $\mathcal{F}$ admits two non-equivalent transversely affine structures, then $\mathcal{F}$ is defined by a closed rational $1$-form, i.e., $\mathcal{F}$ admits a transversely Euclidean structure.
    \item If $\mathcal{F}$ admits two non-equivalent virtually transversely Euclidean structures, then $\mathcal{F}$ admits a rational first integral.
\end{enumerate}
\end{lemma}\subsection{Rational curves and foliations}\label{Tangential foliation} \label{ratcurves-ftang}
Let $X$ be a smooth complex projective variety of dimension $n$. 
A \emph{family of rational curves} $H$ on $X$ is an irreducible component of the scheme $\operatorname{RatCurves}^n(X)$ parametrizing rational curves on $X$.  
We refer to \cite{kollar:rational-on-algebraic} for general results on $\operatorname{RatCurves}^n(X)$ and background on rational curves.  
Denote by $\operatorname{Locus}(H)$ the subset of $X$ swept out by curves parametrized by $H$.  
The family $H$ is said to be \emph{dominating} if $\overline{\operatorname{Locus}(H)} = X$.  
We say that $H$ is \emph{unsplit} if it is proper, and \emph{minimal} if, for a general point $x \in \operatorname{Locus}(H)$, the subscheme $H_x \subset H$ parametrizing curves passing through $x$ is proper.

A smooth projective variety $X$ is called \emph{uniruled} if there is a dominant rational map $Y \times \mathbb{P}^1 \dashrightarrow X$, where $Y$ is a variety of dimension $\dim(X) - 1$. A variety is uniruled if and only if it is covered by rational curves.
Moreover, every uniruled variety admits a minimal dominating family of rational curves (see \cite[Theorem 1.1]{Hwang01}).
From now on, we assume that $X$ is uniruled.
Let $C$ be a rational curve parametrized by a morphism $f\colon \mathbb{P}^1 \longrightarrow X$. We denote by $[C]$ or $[f]$ any point in $H$ corresponding to $C$. By the Birkhoff–Grothendieck theorem, the pullback of the tangent bundle of $X$ by $f$ splits as a direct sum:
\[
f^*T_X \cong \mathcal{O}_{\mathbb{P}^1}(a_1) \oplus \cdots \oplus \mathcal{O}_{\mathbb{P}^1}(a_n),
\]
for some integers $a_1 \ge \cdots \ge a_n$. We say that a rational curve $[f]$ is \emph{free} if all the integers $a_1, \dots, a_n$ are nonnegative. 
Now, let $[f] \in H$ be a general member of a minimal dominating family of rational curves $H$. Then, by \cite[IV 2.9]{kollar:rational-on-algebraic}, we have:
\[
f^*T_X \cong \mathcal{O}_{\mathbb{P}^1}(2) \oplus \mathcal{O}_{\mathbb{P}^1}(1)^{\oplus p} \oplus \mathcal{O}_{\mathbb{P}^1}^{\oplus n-p-1},
\]
where $p = \deg(f^*T_X) - 2 \ge 0$.
\begin{remark}\label{Fano}
A smooth complex projective uniruled variety with Picard number one is Fano.
Indeed, uniruledness implies that \( K_X \) is not nef \cite[Theorem~1]{MM86}; if in addition \( \rho(X) = 1 \), then \( -K_X \) must be ample.
\end{remark}
\subsubsection{Rationally connected quotients}\label{ratconnected}
We now introduce a concept from the theory of rational curves that will help analyze the behavior of a foliation on a smooth projective uniruled variety $X$ with respect to families of rational curves on $X$; see, for instance, Lemma \ref{L:ratconneted} and Proposition \ref{P:deg-maior-que-0}.  
Let $H$ be a family of rational curves on $X$, and let $\overline{H}$ denote the closure of the image of $H$ in $\operatorname{Chow}(X)$. We define an equivalence relation on $X$ as follows. Two points $x, y \in X$ are said to be \emph{$H$-equivalent} if there exists a chain of $1$-cycles from $\overline{H}$ connecting them. By \cite[IV, 4.16]{kollar:rational-on-algebraic}, there exists a proper, surjective, equidimensional morphism
\[
\pi_0 \colon X_0 \longrightarrow T_0
\]
from a dense open subset $X_0 \subset X$ onto a normal variety $T_0$, whose fibers are precisely the $H$-equivalence classes. We call this map the \emph{$H$-rationally connected quotient of $X$}. When $T_0$ is a point, we say that $X$ is \emph{$H$-rationally connected}.  

\begin{remark}\label{R:rat connected}
Note that if $X$ has Picard number one and $H$ is a dominating family of rational curves, then $X$ is $H$-rationally connected. Indeed, let 
\[
\pi_0 \colon X_0 \longrightarrow T_0
\]
be the $H$-rationally connected quotient of $X$, and suppose that $T_0$ is not a point. Observe that by definition, the pre-images under $\pi_0$ of any two distinct points in $T_0$ do not intersect.
By hypothesis, $\operatorname{Pic}(X) \cong \mathbb{Z} \cdot A$ for some ample divisor $A$ on $X$. Let $D$ be an ample divisor on $T_0$ and $Q \in T_0$ a point such that $Q \notin D$. Then 
\(
\overline{\pi_0^{-1}(D)} \sim a \cdot A
\)
for some $a \in \mathbb{Z}_{>0}$, while $\pi_0^{-1}(Q)$ is not contained in $\pi_0^{-1}(D)$. Hence, there exists a curve $\ell \subset \pi_0^{-1}(Q)$ such that $A \cdot \ell = 0$, which is a contradiction, since $A$ generates the Picard group of $X$. See also \cite[Lemmas~6.5 and~6.6]{AD-fano}.
\end{remark}

\begin{lemma}[{\cite[Lemma 6.9]{AD-fano}}]\label{L:ratconneted}
Let $X$ be a smooth projective variety, $H$ an unsplit dominating family of rational curves on $X$, and $\mathcal{F}$ a foliation on $X$. 
Denote by $\pi_0\colon X_0 \longrightarrow T_0$ the $H$-rationally connected quotient of $X$.
If $T_{\mathbb{P}^1} \subset f^*T_{\mathcal{F}}$ for general $[f] \in H$, then
there is an inclusion $T_{X_0/T_0} \subset T_{\mathcal{F}}|_{X_0}$. 
\end{lemma}

\subsubsection{The tangential foliation}\label{S:Ftang}
Let $H$ be a dominating family of rational curves on $X$, and denote by $M \subset \operatorname{Mor}(\mathbb{P}^1, X)$ the irreducible component corresponding to the morphisms that parametrize the curves in $H$. Let $\mathcal{F}$ be a foliation on $X$. Our goal is to introduce a foliation on $M$, known as the \emph{tangential foliation} of $\mathcal{F}$. This construction allows us to study the properties of $\mathcal{F}$ via an analysis of its tangential foliation on $M$; see \cite[Section 6]{loray-pereira-touzet:trivial} and \cite[Section 4]{LPT-def}.  
Consider the evaluation map
\[
\operatorname{ev}\colon M \times \mathbb{P}^1 \longrightarrow X,\qquad ([f],z) \longmapsto f(z).
\]
 
By \cite[I, Theorem 2.16 and II, Proposition 3.5]{kollar:rational-on-algebraic}, the scheme $\operatorname{Mor}(\mathbb{P}^1, X)$ is smooth at a free morphism $f \colon \mathbb{P}^1 \longrightarrow X$, and its tangent space at $[f]$ is canonically identified with $H^0(\mathbb{P}^1, f^*T_X)$. For the foliation $\mathcal{F}$ on $X$, we have a natural inclusion
\(
H^0(\mathbb{P}^1, f^*T_{\mathcal{F}}) \subset H^0(\mathbb{P}^1, f^*T_X).
\)
The tangential foliation on $M$, denoted by $\mathcal{F}_{\operatorname{tang}}$, is defined as follows. For a general free morphism $f \colon \mathbb{P}^1 \longrightarrow X$, the tangent space of $\mathcal{F}_{\operatorname{tang}}$ at $[f]$ is
\[
T_{[f]}\mathcal{F}_{\operatorname{tang}} := H^0(\mathbb{P}^1, f^*T_{\mathcal{F}}).
\]
Alternatively, the tangential foliation is defined in the following way.
Consider the natural projection $\pi\colon M \times \mathbb{P}^1 \longrightarrow M$. By \cite[II Theorem 1.7]{kollar:rational-on-algebraic}, we have an isomorphism $T_M \cong \pi_* \operatorname{ev}^*T_X$, and so the inclusion $T_{\mathcal{F}} \subset T_X$ induces an inclusion
\(
i\colon \pi_* \operatorname{ev}^*T_{\mathcal{F}} \hookrightarrow T_M.
\)
We define the tangential foliation of $\mathcal{F}$ on $M$ by saturating the image of $i$ in $T_M$. By applying \cite[III Corollary 12.9]{hartshorne:ag}, we conclude that these two definitions are equivalent.

By definition, if $L$ is a leaf of $\mathcal{F}_{\operatorname{tang}}$ and $L \times \mathbb{P}^1 \longrightarrow X$ is the restriction of the evaluation morphism to $L \times \mathbb{P}^1$, the pullback of the foliation $\mathcal{F}$ on $X$ via this restriction is either the foliation given by the fibers of the projection $L \times \mathbb{P}^1 \longrightarrow \mathbb{P}^1$ or the foliation with just one leaf. For a more geometric description of the tangential foliation, see \cite[Section 6]{loray-pereira-touzet:trivial}. 
The following theorem illustrates how we can characterize the foliation $\mathcal{F}$ by studying the tangential foliation. 

\begin{theorem}[{\cite[Theorem C]{LPT-def}}]\label{T: teo C LPT}
Let $X$ be a uniruled smooth projective variety and let $\mathcal{F}$ be a codimension-one foliation on $X$.
Fix an irreducible component $M$ of the space of morphisms $\operatorname{Mor}(\mathbb{P}^1,X)$ containing a free morphism and let $\mathcal{F}_{\operatorname{tang}}$ be the tangential foliation of $\mathcal{F}$ defined on $M$. If the general leaf of $\mathcal{F}_{\operatorname{tang}}$ is not algebraic and the general morphism $f\colon \mathbb{P}^1 \longrightarrow X$ in $M$ intersects non-trivially and transversely all the algebraic hypersurfaces invariant by $\mathcal{F}$, then $\mathcal{F}$ is defined by a closed rational $1$-form without divisorial components in its zero set.  
\end{theorem}

When $X$ is a smooth projective uniruled variety, we can establish a version of Lemma \ref{L: pullback tansv} for transversely Euclidean structures in the following context:

\begin{lemma}\label{L:pullback transv euclidean}
Let $X$ be a smooth projective uniruled variety. Let $M$ be an irreducible dominant component of $\operatorname{Mor}(\mathbb{P}^1, X)$. Consider the evaluation morphism 
\[
\operatorname{ev}\colon M \times \mathbb{P}^1 \longrightarrow X, \quad ([f], z) \mapsto f(z).
\]

Let $\mathcal{F}$ be a codimension-one foliation on $X$. Suppose that the general morphism $f\colon \mathbb{P}^1 \longrightarrow X$ in $M$ intersects transversely all the algebraic hypersurfaces invariant by $\mathcal{F}$. The foliation $\operatorname{ev}^* \mathcal{F}$ on $M \times \mathbb{P}^1$ is transversely Euclidean if and only if $\mathcal{F}$ is.
\end{lemma}

\begin{proof}
If $\mathcal{F}$ is transversely Euclidean, pulling back its transverse affine structure makes $\operatorname{ev}^*\mathcal{F}$ transversely Euclidean. Conversely, suppose $\operatorname{ev}^*\mathcal{F}$ is transversely Euclidean; in particular, $\operatorname{ev}^* \mathcal{F}$ is virtually transversely Euclidean. 
By Lemma~\ref{L: pullback tansv}, $\mathcal{F}$ is virtually transversely Euclidean. Hence, there exist holomorphic $1$-forms $\omega_0,\omega_1$ such that $\mathcal{F}$ is defined by $\omega_0$, $d\omega_0=\omega_0\wedge\omega_1$, and $d\omega_1=0$.  The pair $(\operatorname{ev}^* \omega_0, \operatorname{ev}^* \omega_1)$ defines a transverse structure for $\operatorname{ev}^* \mathcal{F}$. The strategy now is to show that the periods of the $1$-forms $\omega_1$ and $\operatorname{ev}^* \omega_1$ coincide.

Let $D \subset X$ be an $\mathcal{F}$-invariant irreducible hypersurface, and consider a closed path $\gamma$ intersecting $D$ transversely (see Section~\ref{S:logarithmic components}). 
Since $M$ is dominant and a general point $[f] \in M$ intersects $D$ transversely, we can choose $[f] \in M$ such that $f$ is an embedding near a point of $D$ and $f(\mathbb{P}^1)$ intersects $D$ exactly once. Let $\Delta \subset \mathbb{P}^1$ be a small disc with $f(\partial\Delta) \subset X \setminus \operatorname{Sing}(\mathcal{F})$ and $f(\Delta)$ intersecting $D$ transversely at one interior point. Set $\gamma' := f(\partial\Delta)$, which is homotopic to $\gamma$, so $\int_{\gamma'}\omega_1 = \int_\gamma \omega_1$.

Consider the lift $\widetilde{\gamma}(t) := ([f], \partial\Delta(t))$ in the fiber $\{[f]\} \times \mathbb{P}^1 \subset M \times \mathbb{P}^1$. Then $\operatorname{ev} \circ \widetilde{\gamma} = \gamma'$ and $\operatorname{ev}$ restricts to an isomorphism on this fiber near $\widetilde{\gamma}$. Thus, 
\[
\int_{\widetilde{\gamma}} \operatorname{ev}^* \omega_1 = \int_{\widetilde{\gamma}} f^* \omega_1 = \int_{\gamma'} \omega_1 = \int_\gamma \omega_1.
\]

Therefore, the periods of $\omega_1$ and $\operatorname{ev}^* \omega_1$ are the same. This implies that $\mathcal{F}$ is also transversely Euclidean.
\end{proof}\subsubsection{The degree of a codimension-one foliation with respect to a family of rational curves}\label{S:deg fol ratcurves}
Let $X$ be a smooth projective uniruled variety and let $H$ be a dominating family of rational curves on $X$. Analogous to the degree of a foliation on $\mathbb{P}^n$, we introduce the concept of the degree of a foliation with respect to $H$.

\begin{definition}
Let $\mathcal{F}$ be a codimension-one foliation on a smooth projective variety $X$ and let $H$ be a dominating family of rational curves on $X$. We define the \emph{degree} of $\mathcal{F}$ with respect to $H$, denoted by $\deg_H \mathcal{F}$, as follows. We set   
$\deg_H \mathcal{F} = -\infty$ if the general curve in $H$ is generically tangent to $\mathcal{F}$. Otherwise, we define $\deg_H \mathcal{F}$ to be the number of tangencies of a general curve in $H$ with $\mathcal{F}$.
\end{definition}

Now let $X$ be a smooth uniruled projective variety with Picard number one; recall that $X$ is Fano (see Remark~\ref{Fano}).  
In this setting, there is no foliation tangent to a general curve in a dominating family of rational curves, as stated in the following result.

\begin{proposition}\label{P:deg-maior-que-0}
Let $X$ be a smooth Fano variety of Picard number one and let $H$ be an unsplit dominating family of rational curves on $X$. Let $\mathcal{F}$ be a codimension-one foliation on $X$. Then 
\[
\deg_H \mathcal{F} \ge 0.
\]
\end{proposition}

\begin{proof}
Let $\pi\colon X_0 \longrightarrow T_0$ be the $H$-rationally connected quotient of $X$, defined in Section \ref{ratconnected}. Since $\rho(X) = 1$, it follows that $T_0$ is a point (see Remark \ref{R:rat connected}). 
Suppose that the general curve of $H$ is tangent to the foliation $\mathcal{F}$ on $X$. Then, by Lemma \ref{L:ratconneted}, there is an inclusion $T_{X_0/T_0} \subset T_{\mathcal{F}}|_{X_0}$, which is a contradiction.
\end{proof}

Let $\mathcal{F}$ be a codimension-one foliation on a smooth uniruled projective variety $X$ induced by a $1$-form $\omega \in H^0(X, \Omega^1_X \otimes N_{\mathcal{F}})$.
Suppose that $\deg_H \mathcal{F} \ge 0$. Consider a general curve $[C] \in H$. Then the restriction $\omega|_C$ belongs to $H^0(C, (\Omega^1_X \otimes N_{\mathcal{F}})|_C)$. Since we have a natural restriction morphism  
\[
(\Omega^1_X \otimes N_{\mathcal{F}})|_C \longrightarrow \Omega^1_C \otimes (N_{\mathcal{F}})|_C,
\]
we can regard $\omega|_C$ as a section of $H^0(C, \Omega^1_C \otimes (N_{\mathcal{F}})|_C)$.
It follows that the degree of $\mathcal{F}$ with respect to $H$ is given by the degree of the line bundle $\Omega^1_C \otimes (N_{\mathcal{F}})|_C$.
Since $C$ is a rational curve, we obtain
\begin{equation}\label{E:normal degree}
\deg_H \mathcal{F} = N_{\mathcal{F}} \cdot C - 2,
\end{equation}
see, for instance, \cite[Proposition 2, Chapter 2]{brunella:book}.

Our next goal is to study codimension-one foliations on $X$ with $\deg_H \mathcal{F} = 0$.

\begin{lemma}\label{L:deg0-fTF}
Let $X$ be a smooth projective uniruled variety, let $H$ be a minimal dominating family of rational curves on $X$, and let $[f] \in M$ be a general free morphism. Let $\mathcal{F}$ be a codimension-one foliation on $X$ with $\deg_H \mathcal{F} = 0$. Then the pullback of $T_{\mathcal{F}}$ via $f$ decomposes as 
\[
f^* T_{\mathcal{F}} \cong \mathcal{O}_{\mathbb{P}^1}(1)^{\oplus p} \oplus \mathcal{O}_{\mathbb{P}^1}^{\oplus (n-1-p)},
\]
where $p = \deg(f^*T_X) - 2$.
\end{lemma}

\begin{proof}
According to \cite[IV Definition 2.8 and Corollary 2.9]{kollar:rational-on-algebraic}, the pullback of the tangent bundle of $X$ via $f$ decomposes as  
\[
f^*T_X \cong \mathcal{O}_{\mathbb{P}^1}(2) \oplus \mathcal{O}_{\mathbb{P}^1}(1)^{\oplus p} \oplus \mathcal{O}_{\mathbb{P}^1}^{\oplus (n-1-p)},
\]  
where $n = \dim(X)$ and $p = \deg(f^*T_X) - 2 \ge 0$. 
Additionally, since $[f] \in M$ is general and $\operatorname{cod} \operatorname{Sing}(\mathcal{F}) \ge 2$, we may assume that $f(\mathbb{P}^1) \cap \operatorname{Sing}(\mathcal{F}) = \emptyset$ (see \cite[II Proposition 3.7]{kollar:rational-on-algebraic}). Thus, the pullback of the tangent sheaf of the foliation is a vector bundle that decomposes as
\[
f^*T_{\mathcal{F}} \cong \bigoplus_{i=1}^{n-1} \mathcal{O}_{\mathbb{P}^1}(a_i).
\]  
Since $\deg_H \mathcal{F} = 0$, it follows that $f^*N_{\mathcal{F}} \cong \mathcal{O}_{\mathbb{P}^1}(2)$.  
Thus, computing the first Chern class, we obtain $c_1(f^*T_{\mathcal{F}}) = \sum_{i=1}^{n-1} a_i = p \ge 0$. Now observe that we have the following exact sequence of the rational curve on $X$ parametrized by $[f]$:  
\[
0 \longrightarrow T_{\mathbb{P}^1} \cong \mathcal{O}(2) \longrightarrow f^* T_X \cong \mathcal{O}(2) \oplus \mathcal{O}(1)^{\oplus p} \oplus \mathcal{O}^{\oplus n-1-p} \longrightarrow N_f \cong \mathcal{O}(1)^{\oplus p} \oplus \mathcal{O}^{\oplus n-1-p} \longrightarrow 0.
\]  
We also have the exact sequence of the foliation:  
\[
0 \longrightarrow f^* T_{\mathcal{F}} \longrightarrow f^* T_X \longrightarrow f^* N_{\mathcal{F}} \cong \mathcal{O}_{\mathbb{P}^1}(2) \longrightarrow 0.
\]  
Since $f(\mathbb{P}^1)$ is transverse to $\mathcal{F}$, we obtain an isomorphism $T_{\mathbb{P}^1} \longrightarrow f^* N_{\mathcal{F}}$, which implies an isomorphism between $f^* T_{\mathcal{F}}$ and $N_f$. Thus,  
\[
f^* T_{\mathcal{F}} \cong \mathcal{O}_{\mathbb{P}^1}(1)^{\oplus p} \oplus \mathcal{O}_{\mathbb{P}^1}^{\oplus n-1-p}.
\]
\end{proof}

In the next result we give a bound on the algebraic rank of $\mathcal{F}$ in terms of the algebraic rank of $\mathcal{F}_{\operatorname{tang}}$.

\begin{proposition}\label{P:bound algebraic rank}
Let $X$ be a smooth projective uniruled variety, and let $\mathcal{F}$ be a codimension-one foliation on $X$ of degree zero with respect to a minimal dominating family of rational curves $H$. Let $M$ be the irreducible component of $\operatorname{Mor}(\mathbb{P}^1, X)$ parametrizing the rational curves in $H$, and consider the tangential foliation $\mathcal{F}_{\operatorname{tang}}$ induced by $\mathcal{F}$ on $M$. Then
\[
r_a(\mathcal{F}) \ge r_a(\mathcal{F}_{\operatorname{tang}}) - p,
\]
where $p = \deg f^*T_X - 2$ for $[f] \in M$.
\end{proposition}

\begin{proof}
Consider the tangential foliation $\mathcal{F}_{\operatorname{tang}}$ induced by $\mathcal{F}$ on $M$. Let $Y \subset M$ be an algebraic variety contained in a general leaf of $\mathcal{F}_{\operatorname{tang}}$ with maximal dimension, and let $z \in \mathbb{P}^1$. Consider the map
\[
\psi \colon Y \longrightarrow X, \quad [f] \mapsto f(z).
\]
By the definition of the tangential foliation, the image of $\psi$ is an algebraic variety contained in a general leaf of $\mathcal{F}$. Therefore, the algebraic rank of $\mathcal{F}$ is greater than or equal to the dimension of the image of $\psi$.
Consider the open subset $M_0 \subset M$ consisting of free morphisms. Fix a point $z \in \mathbb{P}^1$ and a general point $x \in X$ with $f(z) = x$. Define
\[
M_x = \{ [f] \in M_0 \mid f(z) = x \} \subset M_0.
\]
Then
\[
T_{M_x} = \left(\pi|_{M_x \times \mathbb{P}^1}\right)_* \left(\operatorname{ev}|_{M_x \times \mathbb{P}^1}\right)^* T_X\big(-\sigma_z(M_x)\big) \subset T_{M_0}|_{M_x} = \big(\pi_* \operatorname{ev}^* T_X\big)|_{M_x},
\]
where $\sigma_z \colon M \longrightarrow M \times \mathbb{P}^1$ is the section defined by $z$, and the dimension of $M_x$ is
\(
h^0\big(\mathbb{P}^1, f^* T_X(-1)\big) = p + 2,
\)
see \cite[Proposition 2]{Mori1979ProjEM} and \cite[Section 3]{druel2005classes}.
Denote by $Y_x$ the intersection of $Y$ with $M_x$. Observe that $Y_x$ is the fiber of $x$ under the morphism $\psi$. Consider in $M_x$ the foliation $\mathcal{F}_{\operatorname{tang},x}$ defined by  
\[
T_{\mathcal{F}_{\operatorname{tang}}}|_{M_x} \cap T_{M_x}.
\]  
By Lemma \ref{L:deg0-fTF}, this foliation has dimension equal to 
\(
h^0(\mathbb{P}^1, f^* T_{\mathcal{F}}(-1)) = p
\) 
on $M_x$. Note that $Y_x$ is contained in a leaf of $\mathcal{F}_{\operatorname{tang},x}$, and consequently, the dimension of $Y_x$ is at most $p$. Thus, the image of $\psi$ has dimension equal to $r_a(\mathcal{F}_{\operatorname{tang}}) - p$, and we conclude the proof. 
\end{proof}

A direct consequence of Proposition~\ref{P:bound algebraic rank} is that, for a codimension-one foliation of degree zero with respect to a minimal dominating family of rational curves, the algebraic integrability of the tangential foliation $\mathcal{F}_{\operatorname{tang}}$ implies that the foliation $\mathcal{F}$ itself is also algebraically integrable, as established in the following result.

\begin{corollary}\label{L:Ftang alg int}
Let $X$ be a smooth projective uniruled variety, and let $\mathcal{F}$ be a codimension-one foliation on $X$ with degree zero with respect to a minimal dominating family of rational curves $H$. Let $M$ be the irreducible component of $\operatorname{Mor}(\mathbb{P}^1, X)$ parametrizing the rational curves in $H$. If the tangential foliation $\mathcal{F}_{\operatorname{tang}}$ on $M$ is algebraically integrable, then $\mathcal{F}$ is algebraically integrable.
\end{corollary}

\begin{proof}
Let $M$ be the irreducible component in $\operatorname{Mor}(\mathbb{P}^1, X)$ parametrizing curves in $H$.  
Consider the tangential foliation $\mathcal{F}_{\operatorname{tang}}$ induced by $\mathcal{F}$ on $M$, and let $[f] \in M$ be a general free morphism. By Lemma~\ref{L:deg0-fTF}, the rank of $\mathcal{F}_{\operatorname{tang}}$ is given by  
\(
h^0(\mathbb{P}^1, f^* T_{\mathcal{F}}) = n - 1 + p.
\)
Assume that $\mathcal{F}_{\operatorname{tang}}$ is algebraically integrable. Then, by Proposition~\ref{P:bound algebraic rank}, we have  
\[
r_a(\mathcal{F}) \ge r_a(\mathcal{F}_{\operatorname{tang}}) - p = n - 1.
\]
It follows that $r_a(\mathcal{F}) = n - 1$, and therefore $\mathcal{F}$ is algebraically integrable.
\end{proof}

When $\mathcal{F}_{\operatorname{tang}}$ is algebraically integrable, even if $\deg_H \mathcal{F} > 0$, we can still extract some information about the algebraicity of the foliation $\mathcal{F}$, as demonstrated in the following  

\begin{proposition}[{\cite[Proposition 4.10]{LPT-def}}]\label{P:Ftang-algint}
Let $X$ be a uniruled smooth projective variety and let $\mathcal{F}$ be a codimension-one foliation on $X$.
If all leaves of $\mathcal{F}_{\operatorname{tang}}$ are algebraic, then there exists an algebraically integrable foliation $\mathcal{G}$ contained in $\mathcal{F}$ such that 
\[
H^0(\mathbb{P}^1, f^* T_{\mathcal{F}}) = H^0(\mathbb{P}^1, f^* T_{\mathcal{G}})
\]
for a general $f \in M$. In particular, $\mathcal{F}$ is the pullback under a rational map of a foliation on lower-dimensional smooth projective varieties.
\end{proposition}\section{Rational homogeneous varieties}\label{S: Rat homog var}
A complex algebraic variety $X$ is called a \emph{homogeneous variety} if it admits a transitive action of a semisimple linear Lie group $G$. By transitivity, the stabilizers of all points in $X$ are conjugate to a subgroup $P \subset G$. Consequently, the variety can be fully described in terms of this subgroup $P$. The variety $X$ is projective if and only if $P$ is parabolic, that is, a connected subgroup of $G$ containing a Borel subgroup, see \cite[Theorem 7.5]{ottaviani}. We write $X = G/P$ to explicitly indicate the group and the corresponding parabolic subgroup. Moreover, $X$ has Picard number one if and only if $P$ is maximal with respect to inclusion; see \cite{ottaviani} for an introduction to such varieties.

We have a structural theorem for projective rational homogeneous varieties, see \cite[Theorem 1.6]{ottaviani}: 

\begin{theorem}[Borel--Remmert, 1962]
A projective rational homogeneous variety is isomorphic to a product
\[
X = G_1/P_1 \times \cdots \times G_k/P_k,
\]
where $G_i$ are simple groups and $P_i$ are parabolic subgroups.
\end{theorem}

\begin{example}
Let $G = \operatorname{SL}(V)$, where $V = \mathbb{C}^{n+1}$, and let $B \subset G$ be the subgroup of upper-triangular matrices. Then the quotient $X = G/B$ is the variety of complete flags
\[
\{0 \subset V_1 \subset V_2 \subset \cdots \subset V_n \subset V\},
\]
with $\dim(V_r) = r$ for each $r$. More generally, for any parabolic subgroup $P \supset B$, the quotient $X = G/P$ is a flag variety; see \cite[Example 3.1]{BFM}. Grassmannians and projective spaces appear as special cases when $P$ is a maximal parabolic subgroup of $G$.

Now, let $G = \operatorname{SO}(V)$ be the orthogonal group of automorphisms of $V$ preserving a non-degenerate quadratic form $Q$. A Borel subgroup $B \subset G$ is the stabilizer of a complete isotropic flag
\[
\{0 \subset V_1 \subset V_2 \subset \cdots \subset V_n \subset V\},
\]
where each $V_r$ is an isotropic subspace of dimension $r$. The quotient $X = G/B$ is then the variety of such flags:
\[
X = \left\{0 \subset V_1 \subset V_2 \subset \cdots \subset V_n \subset V \;\middle|\; Q(V_n, V_n) = 0 \right\}.
\]
Similarly to the $\operatorname{SL}(V)$ case, other parabolic subgroups of $\operatorname{SO}(V)$ yield additional homogeneous varieties, such as the orthogonal Grassmannian, which parametrizes $k$-dimensional subspaces of $V$ that are isotropic with respect to $Q$. See, for example, \cite[Section 23.3]{fulton-harris:representations}.
\end{example}

Let $G$ be a simple complex linear Lie algebraic group, and let $X$ be a rational homogeneous variety under the action of $G$.  
Denote by $\mathfrak{g}$ the Lie algebra of $G$, and fix a Cartan subalgebra $\mathfrak{h} \subset \mathfrak{g}$.  
Let $\Phi \subset \mathfrak{h}^*$ be the root system of $\mathfrak{g}$, giving the root space decomposition
\[
\mathfrak{g} = \mathfrak{h} \oplus \bigoplus_{\alpha \in \Phi} \mathfrak{g}_\alpha,
\]
where $\mathfrak{h}$ corresponds to the zero weight; see \cite[Section 6]{ottaviani}.  
Choosing a direction in $\mathfrak{h}^*$ relative to the lattice generated by $\Phi$, the roots decompose into positive roots $\Phi^+$ and negative roots $\Phi^-$.  
Let $\{ \alpha_1, \ldots, \alpha_r \} \subset \Phi^+$ be the set of simple roots of $G$, that is, the roots that cannot be expressed as a sum of two positive roots. 
Fix a subset $A \subset \{ \alpha_1, \ldots, \alpha_r \}$ and consider
\[
\Phi^{-}(A) := \left\{ \alpha \in \Phi^{-} \;\middle|\; \alpha = \sum_{\alpha_i \notin A} p_i \alpha_i \right\}.
\]
According to \cite[Proposition-Definition 7.7]{ottaviani}, the parabolic subalgebra corresponding to $A$ is the subalgebra of $\mathfrak{g}$ given by
\[
\mathfrak{p}_A = \mathfrak{h} \oplus \bigoplus_{\alpha \in \Phi^+} \mathfrak{g}_\alpha \oplus \bigoplus_{\alpha \in \Phi^-(A)} \mathfrak{g}_\alpha.
\]
The corresponding parabolic subgroup $P_A \subset G$ is the subgroup whose Lie algebra is $\mathfrak{p}_A$.  
Every parabolic subgroup of $G$ arises in this way; see \cite[Theorem 7.8 and Corollary 7.10]{ottaviani}.  
Consequently, the parabolic subgroups of a simple Lie group can be completely described in terms of the Dynkin diagram of its Lie algebra. Moreover, $P_A$ is \emph{maximal} if the subset $A$ consists of a single root. In this case, $X = G/P_A$ is a Fano variety of Picard rank one; for a full description of the parabolics subgrous, see \cite[Proposition 10.4]{ottaviani}.

\subsection{Homogeneous vector bundles and the Bott--Borel--Weil theorem}\label{S: BBW theorem e hom vec bndles}
Let $X = G/P$ be a projective homogeneous variety, where $G$ is a semisimple complex linear algebraic group and $P \subset G$ is a parabolic subgroup. 
In this section, we introduce homogeneous vector bundles over $X$, describe their construction from irreducible representations of $P$, and state the Bott--Borel--Weil theorem, which relates the cohomology of these bundles to representations of $G$. For further details, see \cite[Theorem 10.16 and Theorem 11.4]{ottaviani}. We begin by defining homogeneous vector bundles.

\begin{definition}
Let $X = G/P$ be a projective homogeneous variety under the action of a semisimple complex linear algebraic group $G$, and let $P \subset G$ be a parabolic subgroup. 
A vector bundle $E$ on $X$ is called \emph{homogeneous}, or \emph{$G$-equivariant}, if the action of $G$ on $X$ lifts to an action on $E$ such that the following diagram commutes:
\[
\begin{tikzcd}
G \times E \arrow[r] \arrow[d] & E \arrow[d] \\
G \times X \arrow[r] & X
\end{tikzcd}
\]
\end{definition}

There exists an equivalence of categories between homogeneous vector bundles on $X = G/P$ and representations of $P$, described as follows. 
Given a representation $V$ of $P$ defined by $\rho\colon P \longrightarrow \operatorname{GL}(V)$, one can construct a vector bundle 
\[
E := G \times_P V
\]
on $X$ by taking the quotient of $G \times V$ via the relation 
\[
(g, v) \sim (g', v') \quad \text{if there exists } p \in P \text{ such that } g = g' p \text{ and } v = \rho(p^{-1}) v'.
\]
The fiber of $E$ over the point stabilized by $P$ is isomorphic to $V$, as shown in \cite[Theorem 9.7]{ottaviani}.  
Hence, to study homogeneous vector bundles on $X = G/P$, it suffices to describe the representations of $P \subset G$; see \cite[Section 9]{ottaviani}.

For the description of irreducible representations of $P$, see \cite[Propositions 10.5 and 10.9]{ottaviani}.  
Recall that an irreducible representation is uniquely determined by its dominant weight, which we denote by $\lambda$. Such a representation will be denoted by $V_\lambda$; see \cite[Theorem 6.36]{ottaviani}.  
The irreducible vector bundle associated with the representation $V_\lambda$ is then denoted by $E_\lambda$.  
Vector bundles on $X$ arising from an irreducible representation of $P$ are called \emph{irreducible}, and those arising from direct sums of irreducible representations of $P$ are called \emph{completely reducible}.

We say that a weight $\lambda$ is \emph{singular} if there exists a positive root $\alpha \in \Phi^+$ such that $(\lambda, \alpha) = 0$.  
On the other hand, $\lambda$ is \emph{regular of index $p$} if it is not singular and there exist exactly $p$ positive roots $\alpha_1, \dots, \alpha_p \in \Phi^+$ such that $(\lambda, \alpha_i) < 0$ for each $i = 1, \dots, p$.  
A weight $\lambda$ is called \emph{dominant} if $(\lambda, \alpha) \ge 0$ for every simple positive root $\alpha$.

The following theorem, known as the Bott--Borel--Weil theorem, is a fundamental tool in our work; see \cite[Theorems 10.11 and 11.4]{ottaviani}.

\begin{theorem}\label{T:BBW}
Let $X = G/P$ be a projective homogeneous variety under the action of a semisimple, simply connected complex Lie group $G$.  
Let $V_\lambda$ be the irreducible representation of $P$ with highest weight $\lambda$, and let $E_{\lambda}$ be the associated irreducible homogeneous vector bundle on $X$.  
Denote by $\delta = \sum_{i=1}^r \lambda_i$ the sum of all fundamental weights of $G$.

\begin{enumerate}
    \item If $\lambda$ is dominant, then $H^0(X, E_\lambda) = V_\lambda$.
    
    \item If $\lambda + \delta$ is singular, then $H^i(X, E_\lambda) = 0$ for all $i$.
    
    \item If $\lambda + \delta$ is regular of index $p$, then $H^i(X, E_\lambda) = 0$ for $i \neq p$.
\end{enumerate}

Moreover,
\[
H^p(X, E_\lambda) = V_{w(\lambda+\delta)-\delta},
\]     
where $w(\lambda + \delta)$ is the unique dominant weight in the fundamental Weyl chamber congruent to $\lambda + \delta$ under the Weyl group action, and $p$ is the length of the Weyl group element sending $\lambda + \delta$ to this chamber (for the definition of the length of an element in the Weyl group, see \cite[Definition 11.10]{ottaviani}).
\end{theorem}\subsection{Adjoint varieties}\label{adjoint-var}
We now focus on a special class of rational homogeneous varieties known as \emph{adjoint varieties}.  
Let $G$ be a connected, simple, reductive Lie group with Lie algebra $\mathfrak{g}$. Consider the adjoint representation of $G$ on $\mathfrak{g}$.    
Consequently, $G$ acts naturally on the complex projective space $\mathbb{P}\mathfrak{g}$.  
Since $\mathfrak{g}$ is simple, there exists a unique closed orbit under this action, denoted by $X(\mathfrak{g})$.  
This orbit is a smooth projective variety in $\mathbb{P}\mathfrak{g}$; see \cite[Proposition 2.6]{B97}.  

We call $X(\mathfrak{g})$ the \emph{adjoint variety} associated with the simple Lie algebra $\mathfrak{g}$, and we refer to the natural embedding $X \hookrightarrow \mathbb{P}\mathfrak{g}$ as the \emph{adjoint embedding}.
Consider the canonical projection $\pi\colon \mathfrak{g} \setminus \{0\} \to \mathbb{P}\mathfrak{g}$.  
Let $v_{\lambda} \in \mathfrak{g}$ be a highest weight vector of the adjoint representation, and let $x = \pi(v_{\lambda})$ be its image under $\pi$.  
Then $X(\mathfrak{g})$ is the orbit of $x$ under the action of $G$; see \cite[Section 2]{Kaji-99}.  
Adjoint varieties always have odd dimension; see \cite[Proposition 2.2]{Kaji-99}.  
Moreover, in a series of two papers \cite{kebekus2000uniqueness, Kebekus2003LinesOC}, Kebekus showed that these varieties are covered by an unsplit minimal dominating family of rational curves.
\subsubsection{The contact structure} \label{S:contact}
Adjoint varieties are equipped with a \emph{contact structure}, which can be described as follows. Let $X$ be an adjoint variety of dimension $n = 2m + 1$.  
There exists an exact sequence  
\[
0 \longrightarrow \mathcal{D} \longrightarrow T_X \xrightarrow{\theta} \mathcal{L} \longrightarrow 0,
\]
where $\mathcal{L}$ is the very ample line bundle defining the embedding $X \hookrightarrow \mathbb{P}(\mathfrak{g})$, and $\mathcal{D}$ is a distribution of rank $2m$. The line bundle $\mathcal{L}$ is called the \emph{contact line bundle}. By a \emph{distribution}, we mean a saturated coherent subsheaf of $T_X$.  
The map $T_X \xrightarrow{\theta} \mathcal{L}$ is defined by contraction with an $\mathcal{L}$-valued $1$-form $\theta$, called a \emph{contact $1$-form}, that is, $\theta \wedge (d\theta)^m$ is nowhere vanishing. In other words, the $\mathcal{O}_X$-bilinear map  
\(
(u,v) \longmapsto \theta([u,v])
\)  
on $\mathcal{D}$ is everywhere non-degenerate; see also \cite{B97}.  
This gives rise to an isomorphism  
\(
\mathcal{D} \longrightarrow \mathcal{D}^{*} \otimes \mathcal{L},
\)  
which sends each section $u$ to $\varphi_u$, where $\varphi_u(v) = \theta([u,v])$; see \cite[Section 1.4.6]{Hwang01}.  Since $\operatorname{rank}(\mathcal{D}) = 2m$, we have:
\[
\det(\mathcal{D}) \cong \det(\mathcal{D}^\vee \otimes \mathcal{L}) \cong \det(\mathcal{D}^\vee) \otimes \mathcal{L}^{2m}.
\]
Consequently,
\(
\det(\mathcal{D})^{\otimes 2} \cong \mathcal{L}^{2m},
\)
and taking first Chern classes we obtain
\(
c_1(\mathcal{D}) = m \mathcal{L}.
\)
Now, from the contact exact sequence,
\[
K_X = \det(T_X^\vee) \cong \mathcal{L}^* \otimes \det(\mathcal{D}^\vee) \cong \mathcal{L}^{-1} \otimes \det(\mathcal{D})^{-1}.
\]
Using $\det(\mathcal{D}) \cong \mathcal{L}^m$ from above, this becomes:
\(
K_X \cong \mathcal{L}^{-1} \otimes \mathcal{L}^{-m} = \mathcal{L}^{-(m+1)}.
\)
Thus, we obtain $-K_X \cong (m+1)\mathcal{L}$ in $\operatorname{Pic}(X)$.
It is conjectured that any smooth projective Fano contact variety of Picard number one is an adjoint variety; see, for instance, the survey in \cite[Section 4.3]{beauville1999riemannian}.

\subsubsection{The adjoint varieties of classical groups}\label{E:adjoint of classical}
\begin{enumerate}

\item \textit{Type $A_n$: a smooth hyperplane section of $\mathbb{P}^n \times \mathbb{P}^n$.}

Let $V$ be an $(n+1)$-dimensional vector space over $\mathbb{C}$. 
Consider the simple Lie group $\operatorname{SL}(V)$ and its Lie algebra of trace-free matrices $\mathfrak{sl}(n+1)$. 
The unique closed orbit of the action of $\operatorname{SL}(V)$ on $\mathbb{P}(\mathfrak{sl}(n+1))$ is the hyperplane section of $\mathbb{P}^n \times (\mathbb{P}^n)^{\vee}$ defined by the trace equation on $\mathfrak{sl}(n+1)$. 
Thus, the adjoint variety $X$ associated to the simple Lie algebra of type $A_n$ is a smooth hyperplane section of $\mathbb{P}^n \times \mathbb{P}^n$. 

Note that the adjoint embedding of $X$ into $\mathbb{P}(\mathfrak{sl}(n+1))$ is the composition of the inclusion of $X$ into $\mathbb{P}^n \times \mathbb{P}^n$ with the Segre embedding:
\[
X \subset \mathbb{P}^n \times \mathbb{P}^n \hookrightarrow \mathbb{P}(\mathfrak{sl}(n+1)).
\] 
We have the two natural projections $\pi_1 \colon X \longrightarrow \mathbb{P}^n$ and $\pi_2 \colon X \longrightarrow \mathbb{P}^n$. For $i = 1,2$, let $h_i$ be the pullback by $\pi_i$ of the hyperplane class. Then 
\[
\mathcal{O}_X(1) = \mathcal{O}_X(h_1 + h_2),
\]
and the canonical bundle of $X$ is given by 
\(
\omega_X = \mathcal{O}_X(-n).
\)

\item \textit{Type $B_n$ and $D_n$: the orthogonal Grassmannian of lines.}

Let us consider the simple Lie group $\operatorname{SO}(V)$ acting on a vector space $V$ of dimension $m$. In the case of $B_n$, we have $m = 2n + 1$, and in the case of $D_n$, we have $m = 2n$. 
The action of $\operatorname{SO}(V)$ on $V$ preserves a symmetric non-degenerate bilinear form $q \in S^2 V$. The adjoint variety $X$ associated with the Lie algebra $\mathfrak{so}(m)$ is the orthogonal Grassmannian 
\[
X = \operatorname{OG}(2,m),
\] 
which parametrizes $2$-dimensional isotropic subspaces $W \subset V$ with respect to this symmetric form $q$. The adjoint embedding  
\[
X \hookrightarrow \mathbb{P}(\mathfrak{so}(m)) = \mathbb{P}\left(\bigwedge^2 V\right)
\]  
is the composition of the inclusion of $X$ into $G(2,m)$ and the Plücker embedding. Due to isomorphisms in the corresponding small-dimensional Lie algebras, we have the following isomorphisms:  
\[
X(\mathfrak{sp}(4)) \cong X(\mathfrak{so}(5)), \quad 
X(\mathfrak{sl}(2)) \cong X(\mathfrak{so}(4)), \quad 
\text{and} \quad X(\mathfrak{sl}(3)) \cong X(\mathfrak{so}(6)).
\]  
Thus, the Picard number of $X(\mathfrak{so}(m))$ is two for $m = 4,6$.  
When the dimension of $V$ is greater than $6$, the adjoint embedding is minimal, meaning that the Picard number of $X$ is generated by $\mathcal{O}_X(1)$.

\item \textit{Type $C_n$: the projective space $\mathbb{P}^{2n-1}$.} 

Let $V$ be a $2n$-dimensional vector space over $\mathbb{C}$. Consider the simple Lie group $\operatorname{Sp}(V)$, which is defined by the $2n \times 2n$ matrices that preserve a fixed non-degenerate skew-symmetric $2$-form on $V$. In an appropriate basis, we identify its Lie algebra $\mathfrak{sp}(2n)$ with $S^2 V$. Then the unique closed orbit of the action of $\operatorname{Sp}(V)$ on $\mathbb{P}(\mathfrak{sp}(2n))$ is $\mathbb{P}(V)$, and the embedding
\[
i\colon \mathbb{P}(V) \hookrightarrow \mathbb{P}(\mathfrak{sp}(2n)) = \mathbb{P}(S^2 V)
\]
is precisely the second Veronese embedding. Hence, the adjoint variety associated to the Lie algebra $\mathfrak{sp}(2n)$ is the second Veronese embedding of the projective space $\mathbb{P}^{2n-1} = \mathbb{P}(V)$. The adjoint embedding is defined by $\mathcal{O}_{\mathbb{P}(V)}(2)$; in other words,
\[
i^*\mathcal{O}_{\mathbb{P}(S^2 V)}(1) \cong \mathcal{O}_{\mathbb{P}(V)}(2),
\]
which follows from the fact that the highest weight of the adjoint representation is $2\lambda_1$.

\end{enumerate}

The adjoint varieties associated with the exceptional simple Lie algebras all have Picard number one, and their adjoint embeddings are minimal. For further details, see \cite[Example 5.5]{Kaji-99}.
\section{Foliations on Fano varieties with Picard number one}\label{sec4-Fano-picard1}
Let $X$ be a Fano variety of Picard number one, and let $H$ be a minimal dominating family of rational curves on $X$ (such a family always exists; see \cite[Chapter V, Theorem 1.6]{kollar:rational-on-algebraic}).  
For a codimension-one foliation $\mathcal{F}$ on $X$, Proposition \ref{P:deg-maior-que-0} implies that the degree of $\mathcal{F}$ with respect to $H$ is non-negative.  
 This section aims to describe those codimension-one foliations on $X$ for which $\deg_H \mathcal{F} = 0$.   
Our first result in this setting is the following.

\begin{proposition}\label{main}
Let $X$ be a smooth Fano variety of Picard number one. Let $\mathcal{F}$ be a codimension-one foliation on $X$ with degree zero with respect to a minimal dominating family of rational curves $H$ on $X$. Then $\mathcal{F}$ is algebraically integrable.
\end{proposition}

\begin{proof}
Let $M$ be the irreducible component in $\operatorname{Mor}(\mathbb{P}^1, X)$ parametrizing curves in $H$. Consider the evaluation map 
\[
\operatorname{ev}\colon M \times \mathbb{P}^1 \longrightarrow X, \qquad (f,z) \longmapsto f(z),
\]
and the natural projection $\pi\colon M \times \mathbb{P}^1 \longrightarrow M$. 
Consider the tangential foliation $\mathcal{F}_{\operatorname{tang}}$ induced by $\mathcal{F}$ on $M$ (see Section \ref{S:Ftang}). When $\mathcal{F}_{\operatorname{tang}}$ is algebraically integrable, Lemma \ref{L:Ftang alg int} implies that $\mathcal{F}$ is algebraically integrable and we are done.  

Suppose that $\mathcal{F}_{\operatorname{tang}}$ is not algebraically integrable. 
Since the Picard number of $X$ is one and $\deg_H \mathcal{F} = 0$, the image of a general morphism $f\colon \mathbb{P}^1 \to X$ in $M$ intersects every invariant hypersurface nontrivially and transversely.  
Then, by Theorem \ref{T: teo C LPT}, $\mathcal{F}$ is defined by a closed rational $1$-form $\omega$ without divisorial components in its zero set and with poles along a simple normal crossing divisor $D = \sum_{i=1}^k D_i$, where each $D_i$ is an irreducible divisor on $X$ and $N_{\mathcal{F}} \cong \mathcal{O}_X(D)$.  
In other words, $\mathcal{F}$ belongs to a logarithmic component (see Section \ref{S:logarithmic components}).

Moreover, since $\deg_H \mathcal{F} = 0$, we have $f^* N_{\mathcal{F}} \cong \mathcal{O}_{\mathbb{P}^1}(2)$.
Let $\mathcal{O}_X(1)$ denote the ample generator of the Picard group of $X$.
As $X$ has Picard number one, the normal sheaf of the foliation is $\mathcal{N}_{\mathcal{F}} \cong \mathcal{O}_X(a)$ for $a = 1$ or $2$. It follows from the description in Section \ref{S:logarithmic components} and from \cite[Theorem 3.1]{Pereira22} that 
\[
\mathcal{N}_{\mathcal{F}} \cong \mathcal{O}_X(D) = \mathcal{O}_X(2).
\]

Consequently, we must have $D = \{h_1 = 0\} + \{h_2 = 0\}$ for some $h_1, h_2 \in H^0(X, \mathcal{O}_X(1))$, and the foliation $\mathcal{F}$ is then defined by a $1$-form with residues $\lambda_1$ along $\{h_1 = 0\}$ and $\lambda_2$ along $\{h_2 = 0\}$, satisfying $\lambda_1 = -\lambda_2$ due to the closedness of $\omega$ (see \cite[Section 3]{Pereira22}).  
In other words, $\mathcal{F}$ is defined by the $1$-form 
\[
\eta = h_1 \, dh_2 - h_2 \, dh_1,
\]
which corresponds to a pencil of sections on $|\mathcal{O}_X(1)|$.
\end{proof}

Assuming additional hypotheses, we show in the next proposition that such degree-zero foliations are pencils of hyperplane sections with respect to the minimal embedding.

\begin{proposition}\label{P: pencil}
Let $\mathcal{F}$ be a codimension-one foliation on a smooth Fano variety $X$ of Picard number one covered by a family of lines $H$. Suppose that there are no integrable $2$-forms in $H^0(X, \Omega^2_X(k))$ for $k \le 2$, and that $\deg_H \mathcal{F} = 0$. Then $\mathcal{F}$ is a pencil of hyperplane sections; that is, $\mathcal{F}$ is defined by a $1$-form of the form $\omega = h_1 \, dh_2 - h_2 \, dh_1$ for $h_1, h_2 \in H^0(X, \mathcal{O}_X(1))$.
\end{proposition}

\begin{proof}
Let $\mathcal{F}$ be a codimension-one foliation on $X$ with $\deg_H \mathcal{F} = 0$. We can apply Proposition \ref{main} to conclude that $\mathcal{F}$ is algebraically integrable. 
Since $\deg_H \mathcal{F} = 0$, there are no tangencies between $\mathcal{F}$ and a general line $[l] \in H$. Assuming that $T_{\mathcal{F}}$ is semistable, we can apply \cite[Proposition 3.5]{LPT13} to conclude that the foliation $\mathcal{F}$ is a pencil of hypersurfaces of degree $m$ with at most two multiple fibers and without non-irreducible fibers (see \cite[Theorem 3.3 and Theorem 3.2]{LPT13}). In other words, $\mathcal{F}$ belongs to a logarithmic component which can be written as $\operatorname{Log}(d_1, d_2)$, where $d_1 \cdot d_2 = m$ and the two multiple fibers have degrees $d_1$ and $d_2$ (see Section \ref{S:logarithmic components}). 

Let $\varphi\colon X \longrightarrow \mathbb{P}^1$ be a rational dominant map defining $\mathcal{F}$. The canonical divisor and consequently the normal bundle of this foliation can be computed as presented in \cite[Paragraph 2.6]{Araujo-Druel-2019}: $K_X \otimes N_{\mathcal{F}} = K_{\mathcal{F}} = K_{X|\mathbb{P}^1} - R(\varphi)$, where $R(\varphi)$ is the ramification divisor and depends only on these two multiple fibers. Then $N_{\mathcal{F}} \cong \mathcal{O}_X(2m + 2 - d_1 - d_2)$. 
Moreover, since $X$ is covered by a family of lines $H$ and $\deg_H \mathcal{F} = 0$, we must have $N_{\mathcal{F}} \cong \mathcal{O}_X(2)$ and then $m = d_1 = d_2 = 1$.
Thus $\mathcal{F}$ is a pencil of hyperplane sections on $X$.

On the other hand, if $T_{\mathcal{F}}$ is not semistable, then, by \cite[Proposition 2.2]{LPT13}, the maximal destabilizing subsheaf of $T_{\mathcal{F}}$ is involutive and defines a subfoliation $\mathcal{G} \subset \mathcal{F}$. According to \cite[Lemma 3.1]{LPT13}, $\mathcal{G}$ is a codimension-two foliation defined by a $2$-form $\eta \in H^0(X, \Omega^2_X(k))$ for $k \le 2$, which contradicts our hypothesis. 
\end{proof}

\subsection{Foliations on adjoint varieties with Picard number one}\label{sec4-picard1}
In this section, we study codimension-one foliations on adjoint varieties of Picard number one. The following example shows that not every adjoint variety of Picard number one is covered by lines. Consequently, Proposition~\ref{P: pencil} does not apply in this setting, making the problem more subtle.

\begin{example}
The adjoint variety associated with the Lie algebra $\mathfrak{sp}(2n)$ is the projective space $\mathbb{P}^{2n-1}$, embedded via the second Veronese map 
\[
\nu_2 \colon \mathbb{P}^{2n-1} \hookrightarrow \mathbb{P}(\mathfrak{sp}(2n)),
\]
for $n \ge 3$. The pullback, via the Veronese embedding, of a codimension-one foliation of degree zero on the ambient space is a codimension-one foliation of degree two on $\mathbb{P}^{2n-1}$.
Note that the space of codimension-one foliations of degree two on $\mathbb{P}^{2n-1}$ has six irreducible components, whereas the space of codimension-one foliations of degree zero on $\mathbb{P}(\mathfrak{sp}(2n))$ has only one. Thus, not every foliation on $\mathbb{P}^{2n-1}$ arises as the pullback of a degree-zero codimension-one foliation on $\mathbb{P}(\mathfrak{sp}(2n))$ under the adjoint embedding.
\end{example}

Note that if $X$ is an adjoint variety with Picard number one and is not isomorphic to the projective space $\mathbb{P}^n$, then $X$ is covered by lines under the adjoint embedding (see \cite[Section 2.3]{kebekus2000uniqueness}). We now show that such varieties satisfy the hypotheses of Proposition \ref{P: pencil}.

\begin{proposition}\label{P: adjoint H^0}
Let $X$ be an adjoint variety of Picard number one not isomorphic to projective space. Then there are no integrable $2$-forms in $H^0(X, \Omega^2_X(k))$ for $k \le 2$.
\end{proposition}
\begin{proof}
Let $X$ be an adjoint variety associated to the simple Lie algebra $\mathfrak{g}$ and suppose that $X$ has Picard number one. Consider the exact sequence associated to the contact structure 
\begin{equation}\label{E:contactdistribuction}   
0 \longrightarrow \mathcal{D} \longrightarrow T_X \longrightarrow \mathcal{L} \longrightarrow 0.
\end{equation}
Here $\mathcal{D}$ is a distribution of rank $2m$ and the map $T_X \longrightarrow \mathcal{L} $ is defined by a contact 1-form $\theta$, see Section \ref{S:contact}.
Since $X$ is not isomorphic to projective space, by \cite[Section 1.4.6]{Hwang01}, the contact line bundle $\mathcal{L}$ is the generator of the Picard group of $X$. From now on we will write $\mathcal{L} = \mathcal{O}_X(1)$.
From (\ref{E:contactdistribuction}), we obtain the exact sequence
\[
0 \longrightarrow \mathcal{D}^\vee(-1) \longrightarrow \Omega^2_X \longrightarrow \bigwedge^2 \mathcal{D}^{\vee} \longrightarrow 0.
\]

Tensoring by $\mathcal{O}_X(k)$, we obtain
\[
\begin{aligned}
(\text{for } k = 1)& \quad 
0 \longrightarrow \mathcal{D}^\vee \longrightarrow \Omega^2_X(1) \longrightarrow \bigwedge^2 \mathcal{D}^\vee(1) \longrightarrow 0, \\
(\text{for } k = 2)& \quad 
0 \longrightarrow \mathcal{D}^\vee(1) \longrightarrow \Omega^2_X(2) \longrightarrow \bigwedge^2 \mathcal{D}^\vee(2) \longrightarrow 0.
\end{aligned}
\]

In \cite[Section 4]{BFM23}, Benedetti, Faenzi, and Marchesi computed the cohomologies of $\bigwedge^i \mathcal{D}^\vee$ and $\bigwedge^i \mathcal{D}^\vee(1)$. Applying these results, we get $H^0(X, \mathcal{D}^\vee) = 0$, $H^1(X, \mathcal{D}^\vee) = H^1(X, \Omega^1_X) \cong \mathbb{C}$, and 
\[
H^0\!\left(X, \bigwedge^2 \mathcal{D}^\vee(1)\right) \cong H^0(X, \mathcal{O}_X) \cong \mathbb{C}.
\]
Hence, we conclude that $H^0(X, \Omega_X^2(1)) = 0$.
Now we want to compute $H^0(X, \Omega^2_X(2))$. Applying \cite[Proposition 4.13--Proposition 4.16]{BFM23}, we obtain an isomorphism 
\[
H^0(X, \Omega^2_X(2)) \cong H^0\!\left(X, \bigwedge^2 \mathcal{D}^\vee(2)\right)
\]
for all adjoint varieties. The next step is to study the homogeneous vector bundle $\bigwedge^2 \mathcal{D}^\vee(2)$ case by case, using the Bott-Borel-Weil Theorem. Throughout the computations, we use Bourbaki's numbering of roots; see \cite{Bou}.

Since $H^0(X, T_X) = H^0(X, \mathcal{L}) = \mathfrak{g}$, the weight corresponding to the homogeneous line bundle $\mathcal{L}$ is the highest weight $\lambda_i$ of the Lie algebra (e.g., $\lambda_2$ for $D_n$; see \cite{Kaji-99} for the highest weights of each simple Lie algebra). Then the weight of $\mathcal{D}$ is $s_i(\lambda_i) = \lambda_i - \alpha_i$, where $s_i$ is the reflection in the Weyl group corresponding to the root $\alpha_i$. 
In the case of the adjoint variety associated to the Lie algebra of type $D_n$, we obtain that the weight corresponding to the homogeneous vector bundle $\mathcal{D}$ is $\lambda_1 - \lambda_2 + \lambda_3$, see \cite[Section 6.2]{BFM}. Then one can use the isomorphism $\mathcal{D} \cong \mathcal{D}^\vee(1)$ to determine the weight of $\mathcal{D}^\vee$.

\begin{enumerate}
    \item First, let $X$ be the adjoint variety associated with the simple Lie algebras of type $B_n$ and $D_n$. 
According to \cite[Section 6.2]{BFM}, the corresponding weight of $\mathcal{D}^\vee$ is $\lambda_1 - 2\lambda_2 + \lambda_3$, which is not dominant. 
Now observe that
\begin{align*}
    \bigwedge^2 \mathcal{D}^\vee &= \bigwedge^2 E_{\lambda_1 + \lambda_3}(-4) \\
    &= \left( \bigwedge^2 E_{\lambda_1} \otimes S^2 E_{\lambda_3} \oplus S^2 E_{\lambda_1} \otimes \bigwedge^2 E_{\lambda_3} \right) \otimes \mathcal{O}_X(-4).
\end{align*}

Using that $\bigwedge^2 E_{\lambda_1} \cong \mathcal{O}_X(1)$ and $S^2 E_{\lambda_3} = E_{2\lambda_3} \oplus \mathcal{O}_X(2)$, we obtain
\(
    \bigwedge^2 \mathcal{D}^\vee \cong E_{2\lambda_1 + 2\lambda_4}(-4) \oplus E_{2\lambda_3}(-3) \oplus \mathcal{O}_X(-1),
\)
and then
\begin{equation*}
    \bigwedge^2 \mathcal{D}^\vee(2) \cong E_{2\lambda_1 + 2\lambda_4}(-2) \oplus E_{2\lambda_3}(-1) \oplus \mathcal{O}_X(1).
\end{equation*}

\item Let $X$ be the adjoint variety associated with the simple Lie algebra of type $E_6$. In the notation of Section \ref{S: BBW theorem e hom vec bndles}, $X = E_6/P(\alpha_2)$. The highest weight of the Lie algebra is $\lambda_2$, and according to \cite[Section 4.4.3]{BFM23} the corresponding highest weight associated to $\mathcal{D}$ is $-2\lambda_2 + \lambda_4$. 
Then one can check the weights of $\bigwedge^2 \mathcal{D} \cong \bigwedge^2 E_{-2\lambda_2 + \lambda_4}$ (for instance, using LiE \cite{vLCL92}) as follows. 

First, identify the highest root of the Lie algebra, which in the case of $X = E_6/P(\alpha_2)$ is $\alpha_2$, and construct a Dynkin diagram $D$ by removing the node corresponding to this root. For instance, for $X = E_6/P(\alpha_2)$, we obtain the following Dynkin diagrams:
 
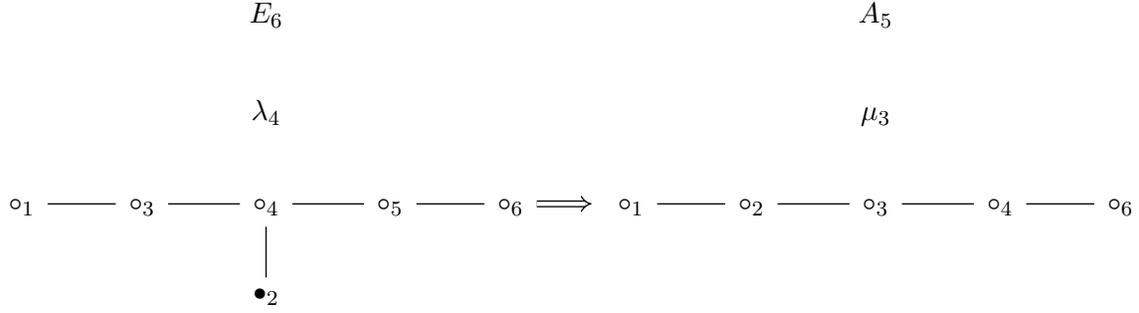
\[
\begin{tikzcd}
    {\circ_1} & {\circ_3} & {\circ_4} & {\circ_5} & {\circ_6} & {E_6} \\
    && {\,\,\bullet_2} \\
    {\circ_1} & {\circ_2} & {\circ_3} & {\circ_4} & {\circ_5} & {A_5}
    \arrow[no head, from=1-1, to=1-2]
    \arrow[no head, from=1-2, to=1-3]
    \arrow[no head, from=1-3, to=1-4]
    \arrow[no head, from=1-3, to=2-3]
    \arrow[no head, from=1-4, to=1-5]
    \arrow[squiggly, from=2-3, to=3-3]
    \arrow[no head, from=3-1, to=3-2]
    \arrow[no head, from=3-2, to=3-3]
    \arrow[no head, from=3-3, to=3-4]
    \arrow[no head, from=3-4, to=3-5]
\end{tikzcd}
\]
The resulting diagram $D$ is the Dynkin diagram of the Levi subgroup $L$ of $P(\alpha_2)$, so $\operatorname{Lie}(L)$ is semisimple (for more details see \cite[Propositions 10.8 and 10.9]{ottaviani}). A homogeneous bundle $E_\lambda$ on $X$ corresponds to an irreducible representation of $L$ whose highest weight we denote by $\mu$, and we denote this representation by $E_\mu$. Here, for $\mathfrak{g}$ of type $E_6$, $E_{\lambda_4}$ corresponds to $F_{\mu_3}$. Then we decompose  
\[
\bigwedge^2 F_{\mu_3} = F_{\mu_2 + \mu_4} \oplus \mathbb{C},
\]  
as a representation of the Lie algebra $\operatorname{Lie}(L)$ (which is of type $A_5$ in the case of $X = E_6/P(\alpha_2)$).  
Hence,  
\[
\bigwedge^2 E_{\lambda_4} = E_{\lambda_3 + \lambda_5}(a) \oplus \mathcal{O}_X(b).
\]  
The values of $a$ and $b$ can be obtained by computing the first Chern class of $\bigwedge^2 V_{\lambda_4}$ and from the fact that $\bigwedge^2 V_{\lambda_4}$ must contain $H^0(X, E_{\lambda_3 + \lambda_5}(a)) \oplus H^0(X, \mathcal{O}_X(b))$.
From this process, we obtain 
\(
\bigwedge^2 E_{-2\lambda_2 + \lambda_4} \cong E_{-3\lambda_2 + \lambda_3 + \lambda_5} \oplus \mathcal{O}_X(-1).
\)
Hence, using that $\mathcal{D}^\vee \cong E_{-2\lambda_1 + \lambda_3}$, we obtain
\[
\bigwedge^2 \mathcal{D}^\vee(2) \cong E_{-\lambda_2 + \lambda_3 + \lambda_5} \oplus \mathcal{O}_X(1).
\]
In each of the cases below, we follow the same approach as for $E_6$.

\item Let $X$ be the adjoint variety associated with the simple Lie algebra of type $E_7$. Similarly to the case of $E_6$, the highest weight of the Lie algebra is $\lambda_1$, and according to \cite[Section 4.4.4]{BFM23} the corresponding highest weight associated to $\mathcal{D}$ is $-2\lambda_1 + \lambda_3$. 
One can check that
\(
\bigwedge^2 E_{-2\lambda_1 + \lambda_3} \cong E_{\lambda_4 - 3\lambda_1} \oplus \mathcal{O}_X(-1).
\)
Hence, using that $\mathcal{D}^\vee \cong E_{-2\lambda_1 + \lambda_3}$, we obtain
\[
\bigwedge^2 \mathcal{D}^\vee(2) = E_{\lambda_4 - \lambda_1} \oplus \mathcal{O}_X(1).
\]

\item Let $X$ be the adjoint variety associated with the simple Lie algebra of type $E_8$. The highest weight of the Lie algebra is $\lambda_8$, and according to \cite[Section 4.4.5]{BFM23} the corresponding highest weight associated to $\mathcal{D}$ is $-2\lambda_8 + \lambda_7$.
One can check that
\(
\bigwedge^2 E_{-2\lambda_8 + \lambda_7} \cong E_{\lambda_6 - 3\lambda_1} \oplus \mathcal{O}_X(-1).
\)
Hence, using that $\mathcal{D}^\vee \cong E_{-2\lambda_8 + \lambda_7}$, we obtain
\[
\bigwedge^2 \mathcal{D}^\vee(2) \cong E_{\lambda_6 - \lambda_1} \oplus \mathcal{O}_X(1).
\]

\item Let $X$ be the adjoint variety associated with the simple Lie algebra of type $F_4$. The highest weight of the Lie algebra is $\lambda_1$, and according to \cite[Section 4.4.2]{BFM23} the corresponding highest weight associated to $\mathcal{D}$ is $-2\lambda_1 + \lambda_2$.
One can check that
\(
\bigwedge^2 E_{-2\lambda_1 + \lambda_2} \cong E_{-3\lambda_1 + 2\lambda_3} \oplus \mathcal{O}_X(-1).
\)
Hence, using that $\mathcal{D}^\vee \cong E_{-2\lambda_1 + \lambda_3}$, we obtain
\[
\bigwedge^2 \mathcal{D}^\vee(2) \cong E_{-\lambda_1 + 2\lambda_3} \oplus \mathcal{O}_X(1).
\]

\item Let $X$ be the adjoint variety associated with the simple Lie algebra of type $G_2$. The highest weight of the Lie algebra is $\lambda_2$, the vector bundle $\mathcal{O}_X(1)$ corresponds to the weight $\lambda_1$, and according to \cite[Section 4.4.1]{BFM23} the corresponding highest weight associated to $\mathcal{D}$ is $3\lambda_1 - 2\lambda_2$.
One can check that
\(
\bigwedge^2 E_{3\lambda_1 - 2\lambda_2} \cong E_{4\lambda_1 - 3\lambda_2} \oplus \mathcal{O}_X(-1).
\)
Hence, using that $\mathcal{D}^\vee = E_{3\lambda_1 - 2\lambda_2}$, we obtain
\[
\bigwedge^2 \mathcal{D}^\vee(2) \cong E_{4\lambda_1 - \lambda_2} \oplus \mathcal{O}_X(1).
\]

\end{enumerate}

Therefore, using the isomorphism $H^0(X, \Omega^2_X(2)) \cong H^0\!\left(X, \bigwedge^2 \mathcal{D}^\vee(2)\right)$ and applying the Bott--Borel--Weil theorem (Theorem~\ref{T:BBW}), we conclude that for every adjoint variety $X$ with Picard number one and $X \not\cong \mathbb{P}^n$, there exists an isomorphism
\[
H^0\!\left(X, \bigwedge^2 \mathcal{D}^\vee(2)\right) \cong H^0(X, \mathcal{O}_X(1)).
\]
More concretely, this can be described as
\[
\phi \colon H^0(X, \Omega^2_X(2)) \longrightarrow H^0\!\left(X, \bigwedge^2 \mathcal{D}^\vee(2)\right) \cong H^0(X, \mathcal{O}_X(1)) \cdot \gamma,
\]
where $\gamma$ is a nonzero section of $H^0\!\left(X, \bigwedge^2 \mathcal{D}^\vee(1)\right) \cong \mathbb{C}$. Let $\psi_\theta \in H^0\!\left(X, \bigwedge^2 \mathcal{D}^\vee(1)\right)$ be given by $\psi_\theta(u, v) = \theta([u, v])$ for $u, v$ sections of $T_X$, where $\theta$ is the contact $1$-form on $X$.
Note that the $2$-form $\psi_\theta$ is not locally decomposable. Indeed, $\psi_\theta$ is locally decomposable if and only if $\mathcal{K} := \ker(\iota_{\psi_\theta})$ has rank two, where the map $\iota_{\psi_\theta}$ is given by
\[
\iota_{\psi_\theta} \colon T_X \longrightarrow \Omega_X^1(1), \qquad u \longmapsto u \,\iota\, \psi_\theta = \theta([u, \cdot]).
\]

Since $\mathcal{D}$ is the contact distribution, $\iota_{\psi_\theta}|_{\mathcal{D}}$ is injective, and then $\mathcal{D} \cap \mathcal{K} = \{0\}$. As $\operatorname{rank}(\mathcal{D}) = \dim(X) - 1$, this implies that $\operatorname{rank}(\mathcal{K}) \le 1$. Thus, $\psi_\theta$ is not locally decomposable.
Note that the isomorphism $\phi$ is given by the wedge product of the dualization of the inclusion $\mathcal{D} \longrightarrow T_X$.
Consequently, the image under $\phi$ is given by its restriction to $\bigwedge^2 \mathcal{D}^\vee \otimes \mathcal{O}_X(2)$.

Now, suppose that a $2$-form $\eta \in H^0(X, \Omega^2_X(2))$ is integrable. In particular, this means that $\eta$ is locally decomposable, that is, locally we can write $\eta = \eta_1 \wedge \eta_2$, where $\eta_i$, $i = 1, 2$, are $1$-forms defined on an open set of $X$. Applying $\phi$, we have $\phi(\eta) = \eta_1' \wedge \eta_2'$, where $\eta_i'$ is the restriction of $\eta_i$ to $\mathcal{D}^\vee \otimes \mathcal{O}_X(1)$. This would imply that a $2$-form in $H^0\!\left(X, \bigwedge^2 \mathcal{D}^\vee(2)\right)$ is locally decomposable, which contradicts the fact that $\iota_{\psi_\theta}$ is not locally decomposable.
\end{proof}

\begin{corollary}\label{corollary one comp}
Let $\mathcal{F}$ be a codimension-one foliation on an adjoint variety $X$ with Picard number one, where $X \not\cong \mathbb{P}^n$. Let $H$ be a family of lines in $X$. Suppose that $\deg_H \mathcal{F} = 0$. Then $\mathcal{F}$ is induced by a pencil of hyperplane sections under the minimal embedding. Hence, the space of codimension-one foliations on $X$ of degree $0$ with respect to $H$ has exactly one irreducible component.
\end{corollary}

\begin{proof}
By Proposition \ref{P: adjoint H^0}, $H^0(X, \Omega^2_X(k))$ has no integrable $2$-forms for $k \le 2$. 
Since $X$ is covered by lines (see \cite{kebekus2000uniqueness, Kebekus2003LinesOC}), Proposition \ref{P: pencil} implies that every degree-zero codimension-one foliation on $X$ is the restriction under the minimal embedding $i\colon X\hookrightarrow\mathbb{P}^N$ of a pencil of hyperplane sections from the ambient projective space $\mathbb{P}^N$. 
The space of such foliations on $\mathbb{P}^N$ is identified with the Grassmannian of lines in the dual projective space (see \cite[Theorem 4.3]{ACM:FanoDist}), hence irreducible. The result follows.
\end{proof}

\begin{remark}
Let $X = \operatorname{OG}(2, n+1)$ be the orthogonal Grassmannian embedded in the projective space $\mathbb{P}^N$ by the Plücker embedding.
In \cite[Section 6]{BFM}, Benedetti, Faenzi, and Muniz proved that the ideals defining the space of codimension-one foliations of degree zero with respect to a family of lines on $X$ and the space of codimension-one foliations of degree zero on $\mathbb{P}^N$ are different. However, Corollary \ref{corollary one comp} implies that every codimension-one foliation on $X$ of degree zero with respect to a family of lines is induced by a degree-zero foliation on $\mathbb{P}^N$.
\end{remark}\section{Foliations on adjoint varieties with Picard number two}\label{sec5-Picard2}
In this section, we consider the adjoint variety of type $A_n$. Recall from Section~\ref{E:adjoint of classical} that $X = X(\mathfrak{sl}(n+1))$ is a smooth hyperplane section of $\mathbb{P}^n \times \mathbb{P}^n \subset \mathbb{P}(\mathfrak{sl}(n+1))$, and that $X$ is equipped with two natural projections:
\[
\begin{tikzcd}[column sep=scriptsize]
	& X \\
	\mathbb{P}^n & {} & \mathbb{P}^n
	\arrow["{\pi_1}"{description}, from=1-2, to=2-1]
	\arrow["{\pi_2}"{description}, from=1-2, to=2-3]
\end{tikzcd}
\]

There are two minimal dominating families of rational curves on $X$:
\begin{enumerate}
    \item The first consists of lines contained in the fibers of $\pi_1$, denoted by $H_1$.
    \item The second consists of lines contained in the fibers of $\pi_2$, denoted by $H_2$.
\end{enumerate}

Our goal is to study codimension-one foliations of minimal degree with respect to these two families.

\subsection{Foliations tangent to every curve in a family of rational curves}

Let $X = X(\mathfrak{sl}(n+1))$ be a smooth hyperplane section of $\mathbb{P}^n \times \mathbb{P}^n$. Using the notation introduced above, we obtain the following result.

\begin{proposition}\label{deg infinito}
Let $\mathcal{F}$ be a codimension-one foliation on $X$. If $\deg_{H_1}(\mathcal{F}) = -\infty$, then $\mathcal{F}$ is the pullback via $\pi_1$ of a codimension-one foliation on $\mathbb{P}^n$ of degree $\deg_{H_2}(\mathcal{F})$.
\end{proposition}

\begin{proof}
By definition, $\deg_{H_1}(\mathcal{F}) = -\infty$ implies that for a general $\ell_1 \in H_1$, the line $\ell_1$ is contained in a leaf of the foliation $\mathcal{F}$. Therefore, the relative tangent sheaf $T_{\pi_1}$ is contained in the tangent sheaf $T_{\mathcal{F}}$. Hence, there exists a foliation $\mathcal{G}$ on $\mathbb{P}^n$ such that $\mathcal{F} = \pi_1^*\mathcal{G}$ (see \cite[Section~2.3 and Lemma~2.4]{loray-pereira-touzet:trivial}).

We now prove that the degree of $\mathcal{G}$ is equal to $\deg_{H_2}(\mathcal{F})$. Let $\ell \subset \mathbb{P}^n$ be a general line. Since $\operatorname{codim} \operatorname{Sing}(\mathcal{G}) \ge 2$, it follows from \cite[II, Proposition~3.7]{kollar:rational-on-algebraic} that $\ell$ does not intersect the singular locus of $\mathcal{G}$. Moreover, there exists a line $\ell_2 \in H_2$ such that $\pi_1(\ell_2) = \ell$. By the generality of $\ell$, we may assume that $\ell_2$ is also general. Since $\mathcal{F} = \pi_1^*\mathcal{G}$, the degree of $\mathcal{G}$, which equals the number of tangencies of $\ell$ with $\mathcal{G}$, coincides with the number of tangencies of $\mathcal{F}$ with the line $\ell_2$. Therefore, $\deg(\mathcal{G}) = \deg_{H_2}(\mathcal{F})$.
\end{proof}

Similarly, if $\deg_{H_2} \mathcal{F} = -\infty$, then $\mathcal{F}$ is the pullback via $\pi_2$ of a codimension-one foliation on $\mathbb{P}^n$ of degree equal to $\deg_{H_1} \mathcal{F}$.

\begin{remark}
There is no codimension-one foliation $\mathcal{F}$ on $X$ such that $\deg_{H_i}(\mathcal{F}) = -\infty$ for both $i = 1,2$. Indeed, suppose that such a foliation exists on $X$. Then, by applying Proposition~\ref{deg infinito}, we conclude that there exists a codimension-one foliation on $\mathbb{P}^n$ of degree $-\infty$, which contradicts Proposition~\ref{P:deg-maior-que-0}.
\end{remark}

\subsection{Foliations with degree zero with respect to both families of rational curves}
Let $X$ be a smooth hyperplane section of $\mathbb{P}^n \times \mathbb{P}^n$. Our goal in this section is to classify codimension-one foliations on $X$ of degree zero with respect to both minimal families of rational curves. Observe that, according to Equation~\ref{E:normal degree}, the normal bundle of such foliations is
\[
\mathcal{N} \cong \mathcal{O}_X(2h_1 + 2h_2) \cong \mathcal{O}_X(2),
\]
where each $h_i$ is the pullback of the hyperplane class via $\pi_i$ for $i = 1,2$.

Let $i\colon X \hookrightarrow \mathbb{P}(\mathfrak{sl}(n+1)) = \mathbb{P}^N$ be the adjoint embedding, which is the composition of the inclusion $X \hookrightarrow \mathbb{P}^n \times \mathbb{P}^n$ with the Segre embedding $s\colon \mathbb{P}^n \times \mathbb{P}^n \longrightarrow \mathbb{P}^{N+1}$. Note that the restriction of a degree-zero codimension-one foliation on $\mathbb{P}^N$ gives a codimension-one foliation on $X$ with normal sheaf
\[
i^*\mathcal{O}_{\mathbb{P}^N}(2) = \mathcal{O}_X(2).
\]
Thus, this foliation has degree zero with respect to both families of minimal rational curves on $X$. The natural question is whether all foliations on $X$ of degree zero with respect to both families arise as pullbacks of degree-zero foliations on $\mathbb{P}^N$.

We begin by establishing the following structure result.
\begin{proposition}\label{Picard 2-Structural}
Let $X$ be a smooth hyperplane section of $\mathbb{P}^n \times \mathbb{P}^n$ for $n \ge 3$.
Let $\mathcal{F}$ be a codimension-one foliation on $X$, and denote by $H_1$ and $H_2$ the two families of lines on $X$. Suppose that $\deg_{H_i} \mathcal{F} = 0$ for $i = 1,2$. Then either:
\begin{enumerate}
    \item $\mathcal{F}$ is defined by a closed rational $1$-form without divisorial components in its zero set; or
    \item $\mathcal{F}$ is algebraically integrable.
\end{enumerate}
\end{proposition}

\begin{proof}
Let $\mathcal{F}$ be a codimension-one foliation on $X$ with $\deg_{H_i} \mathcal{F} = 0$ for $i = 1,2$. 
Let $M_i$ be the irreducible component of $\operatorname{Mor}(\mathbb{P}^1, X)$ parametrizing curves in $H_i$, for $i = 1,2$. Let $\mathcal{F}_{\operatorname{tang},i}$ be the tangential foliation of $\mathcal{F}$ on $M_i$, for $i = 1,2$. 
If $\mathcal{F}_{\operatorname{tang},i}$ is algebraically integrable for some $i \in \{1,2\}$, then by Lemma \ref{L:Ftang alg int}, we conclude that $\mathcal{F}$ is algebraically integrable. 

Now suppose that $\mathcal{F}_{\operatorname{tang},i}$ is not algebraically integrable for $i = 1,2$. According to \cite[Theorem B]{LPT-def}, $\mathcal{F}$ admits a transversely projective structure. Our goal now is to describe this transverse structure for $\mathcal{F}$.
Consider the isomorphism $X \cong \mathbb{P}\Omega^1_{\mathbb{P}^n}$, let $\pi_1\colon X \longrightarrow \mathbb{P}^n$ be the natural $\mathbb{P}^{n-1}$-bundle and let $\pi_2\colon X \longrightarrow (\mathbb{P}^n)^* \cong \mathbb{P}^n$ be the projection map that sends a point in the projectivized cotangent bundle to the corresponding dual space, which is identified with $\mathbb{P}^n$.
Let $F_i \cong \mathbb{P}^{n-1}$ be the general fiber of $\pi_i$, for $i = 1,2$.

Since the normal bundle of $\mathcal{F}$ is $\mathcal{O}_X(2,2)$, $T_{\mathcal{F}}|_{F_i} \cap T_{F_i}$ gives a foliation of degree zero on $F_i \cong \mathbb{P}^{n-1}$, that is, a pencil of hyperplanes on $F_i \cong \mathbb{P}^{n-1}$ defined by a linear projection $\pi\colon \mathbb{P}^{n-1} \dashrightarrow \mathbb{P}^1$.
Therefore, by \cite[Paragraph 7.8]{AD-fano}, $\mathcal{F}$ admits the following description. For $i = 1,2$ and for an open subset $\mathcal{U}_i$ of $\mathbb{P}^n$ such that $\operatorname{codim}_{\mathbb{P}^n}(\mathbb{P}^n \setminus \mathcal{U}_i) \ge 2$, there exists a rank $n-2$ subbundle $\mathcal{V}_i \subset T_{\mathcal{U}_i}$ and a rank-$2$ bundle $\mathcal{K}_i$ defined as the kernel of the map ${\Omega^1_{\mathbb{P}^n}}|_{\mathcal{U}_i} \longrightarrow {\mathcal{V}_i^*}|_{\mathcal{U}_i}$. We have the exact sequence over $\mathcal{U}_i$:
\[
0 \longrightarrow {\mathcal{K}_i} \longrightarrow {\Omega^1_{\mathbb{P}^n}}|_{\mathcal{U}_i} \longrightarrow {\mathcal{V}_i^*}|_{\mathcal{U}_i} \longrightarrow 0.
\]

We consider the $\mathbb{P}^1$-bundle $Z_i = \mathbb{P}_{\mathcal{U}_i} \mathcal{K}_i$ over $\mathcal{U}_i$ with a natural projection $q_i\colon Z_i \longrightarrow \mathcal{U}_i$. Consider $\mathcal{U} = \pi_1^{-1}(\mathcal{U}_1) \cap \pi_2^{-1}(\mathcal{U}_2)$, an open subset of $X$. Then for $i = 1,2$, we have a rational map $p_i\colon \mathcal{U} \dashrightarrow Z_i$ (see \cite[Proposition 7.10 and Theorem 9.2]{AD-fano}) fitting into the following diagram:

\[
\begin{tikzcd}
		{Z_1} &  \mathcal{U}& {Z_2} \\
	\mathcal{U}_1 && \mathcal{U}_2
	\arrow["{q_1}"', from=1-1, to=2-1]
	\arrow["{p_1}"', dashed, from=1-2, to=1-1]
	\arrow["{p_2}", dashed, from=1-2, to=1-3]
	\arrow["{\pi_1}", from=1-2, to=2-1]
	\arrow["{\pi_2}"', from=1-2, to=2-3]
	\arrow["{q_2}", from=1-3, to=2-3]
\end{tikzcd}
\]

The foliation $\mathcal{F}$ restricted to $\mathcal{U}$ is the pullback via $p_i$ of a foliation $\mathcal{G}_i$ on $Z_i$, for $i = 1,2$.  
The next step is to show that the transverse structure of $\mathcal{F}$ arises from transverse structures of each $\mathcal{G}_i$ (see Lemma~\ref{L: pullback tansv}), then prove that these structures are not equivalent, and finally apply Lemma~\ref{L:2transv aff structures} to conclude that this transverse structure is transversely Euclidean.  
Hence, $\mathcal{F}$ is defined by a closed rational $1$-form.

Let $\overline{H_i}$ be the family of rational curves on $Z_i$ that are images of general curves from $H_j$ via $p_j$, where $\{i,j\} = \{1,2\}$. Note that these curves are mapped to lines on $\mathbb{P}^n$ contained in $\mathcal{U}_i$ via $q_j$. 
Now, let $\ell_2 \in H_2$ be a general curve in $H_2$, so that the image of $\ell_2$ by $\pi_1$ is a general line on $\mathbb{P}^n$ contained in $\mathcal{U}_i$.
Since the indeterminacy locus of $p_1$ has codimension two, it does not intersect $\ell_2$. Moreover, since $\mathcal{F}$ is regular in a neighborhood of $\ell_2$ and transverse to $\ell_2$, its image $\overline{\ell_2}$ on $Z_1$ satisfies $\overline{\ell_2} \cap \operatorname{Sing}(\mathcal{G}_1) = \emptyset$ and $\overline{\ell_2}$ is transverse to $\mathcal{G}_1$. We can proceed similarly for the foliation $\mathcal{G}_2$ and then conclude that $\mathcal{G}_i$ is a degree-zero foliation with respect to the family $\overline{H_i}$.

Observe that the map $q_i$ is a $\mathbb{P}^1$-bundle over $\mathcal{U}_i$ and since the general fiber of $q_i$ is transverse to $\mathcal{G}_i$, the foliation $\mathcal{G}_i$ is a Riccati foliation. Consequently, $\mathcal{G}_i$ admits a transversely projective structure. Applying Lemma \ref{L: pullback tansv}, we conclude that $\mathcal{F}$ has two transversely projective structures induced by the transverse structures of $\mathcal{G}_i$, for $i = 1,2$.

Since these structures arise from distinct fibrations, they are not equivalent, and Lemma \ref{L:2transv aff structures} implies that $\mathcal{F}$ is virtually transversely Euclidean. Consequently, for a $1$-form $\omega_0$ defining the foliation $\mathcal{F}$, there exists a $1$-form $\omega_1$, where $\omega_1$ is a closed logarithmic $1$-form with periods commensurable to $\pi \sqrt{-1}$. Note that if the periods of $\omega_1$ are integral multiples of $2\pi \sqrt{-1}$, then $\exp \left( \int \omega_1 \right)$ is well-defined globally and 
\[
\exp \left( \int \omega_1 \right) \omega_0
\]
is the desired closed rational $1$-form (see Section \ref{S: tansversality of fol}). 

Assume from now on that the periods of $\omega_1$ are not integral multiples of $2\pi\sqrt{-1}$. According to \cite[Theorem 4.1 and Proposition 4.4]{LPT-def}, the foliation $\operatorname{ev}^* \mathcal{F}$ is defined by a closed rational $1$-form, that is, $\operatorname{ev}^* \mathcal{F}$ is transversely Euclidean. Lemma \ref{L:pullback transv euclidean} implies that $\mathcal{F}$ is transversely Euclidean as well. This concludes the proof that $\mathcal{F}$ is defined by a closed rational $1$-form $\omega$.

Suppose that $\omega$ has a divisorial component in its zero set. Since the restriction $\omega|_{F_i}$ induces a foliation of degree zero on $F_i \cong \mathbb{P}^{n-1}$, such a divisorial component must be of the form $\pi_i^{-1}(D_i)$ for some divisor $D_i \subset \mathbb{P}^n$. However, this cannot happen simultaneously for $i = 1,2$. Therefore, the closed rational $1$-form $\omega$ defining $\mathcal{F}$ does not have any divisorial components in its zero set.
\end{proof}

\begin{proposition}\label{algint-picard2}
Let $X$ be a smooth hyperplane section of $\mathbb{P}^n \times \mathbb{P}^n$ and let $H_1$ and $H_2$ be the two families of lines on $X$.
Let $\mathcal{F}$ be a codimension-one foliation on $X$ with normal sheaf $\mathcal{O}_X(2)$, that is, $\deg_{H_i} \mathcal{F} = 0$ for $i = 1,2$. If $n \ge 3$ and $\mathcal{F}$ is algebraically integrable, then $\mathcal{F}$ is given by a pencil of hyperplane sections. 
\end{proposition}

\begin{proof}
Consider again the two natural projections 
\[
\begin{tikzcd}[column sep=scriptsize]
	& X \\
	{\mathbb{P}^n} & {} & {\mathbb{P}^n}
	\arrow["{\pi_1}"{description}, from=1-2, to=2-1]
	\arrow["{\pi_2}"{description}, from=1-2, to=2-3]
\end{tikzcd}
\]
Then $\operatorname{Pic}(X) = \mathbb{Z}[h_1] \oplus \mathbb{Z}[h_2]$, where each $h_i$ is the pullback by $\pi_i$ of the hyperplane class, $i = 1,2$.
Let $\mathcal{F}$ be an algebraically integrable foliation on $X$ with normal sheaf $\mathcal{N}_{\mathcal{F}} \cong \mathcal{O}_X(2)$, and let $Z \subset X$ be an invariant algebraic hypersurface. Then there exist integers $a, b \in \mathbb{Z}$ such that 
\[
[Z] = a[h_1] + b[h_2] \in \operatorname{Pic}(X).
\]

As noted before, the restriction of $\mathcal{F}$ to a general fiber $F_i \cong \mathbb{P}^{n-1}$ of $\pi_i$ (for $i = 1,2$) is a degree-zero foliation of rank $n-2$, and thus $Z \cap F_i$ is a hyperplane section of $\mathbb{P}^{n-1}$; note that here we use that $n \ge 3$.  
Consequently, we obtain $a = b = 1$.  
Since the general leaf of $\mathcal{F}$ is algebraic, it follows that $\mathcal{F}$ is defined by a pencil of hyperplane sections.
\end{proof}

Aiming to study codimension-one foliations on $X$ of degree zero with respect to both families of lines and considering the result of Proposition \ref{Picard 2-Structural}, our next objective is to examine foliations defined by closed rational $1$-forms without divisorial components in their zero sets. Recall from Section \ref{S:logarithmic components} that these foliations are contained in the closure of logarithmic components. This leads to the following result:

\begin{theorem}\label{P: logarithmic on PnxPn}
Let $X$ be a smooth hyperplane section of $\mathbb{P}^n \times \mathbb{P}^n$.  
There are exactly two logarithmic components in the space of codimension-one foliations on $X$ with normal sheaf $N_{\mathcal{F}} \cong \mathcal{O}_X(2)$; that is, foliations of degree zero with respect to both minimal families of rational curves.

Furthermore, when $n \ge 3$, these are the only components. In other words, the space of codimension-one foliations on $X$ whose normal sheaf is $\mathcal{N} = \mathcal{O}_X(2)$ has exactly two irreducible components.

\end{theorem}

\begin{proof}
Since we are considering codimension-one foliations with normal sheaf $N_{\mathcal{F}} \cong \mathcal{O}_X(2)$, any such foliation belonging to a logarithmic component is defined by a $1$-form that has no zeros in codimension one and has poles on the support of a simple normal crossing divisor $D = \sum_{i=1}^{k} D_i$, where each $D_i$ is an irreducible divisor on $X$. Note that the normal bundle satisfies $\mathcal{O}_X(2) \cong \mathcal{O}_X\bigl(\sum D_i\bigr)$ (see Section~\ref{S:logarithmic components}). 

Recall that we have the two natural projections $\pi_1 \colon X \longrightarrow \mathbb{P}^n$ and $\pi_2 \colon X \longrightarrow \mathbb{P}^n$. For $i = 1,2$, let $h_i$ denote the pullback via $\pi_i$ of the hyperplane class on $\mathbb{P}^n$. 
Consequently, the $1$-form defining $\mathcal{F}$ has poles along a divisor $D = \sum_{i=1}^k D_i$ such that either:
\begin{enumerate}
    \item $k = 2$, $D_1, D_2 \in |\mathcal{O}_X(1)|$, and the residues of this $1$-form along each $D_i$ satisfy $\lambda_1 = -\lambda_2$; or  
    \item $k = 4$, $D_1, D_2 \in |h_1|$, $D_3, D_4 \in |h_2|$, and the residues satisfy $\lambda_1 = -\lambda_2$ and $\lambda_3 = -\lambda_4$.  
\end{enumerate}
Observe that we could also have $k = 3$, with $D_1 \in |\mathcal{O}_X(1)|$, $D_2 \in |h_1|$, $D_3 \in |h_2|$, and the residues satisfying $\lambda_1 = -\lambda_2 = -\lambda_3$. However, this condition on the residues implies that this case is contained in the case where $k = 2$. Thus these are the only cases, and we will denote the logarithmic component in case (1) by $\operatorname{Log}(1,1)$ and the logarithmic component in case (2) by $\operatorname{Log}\bigl((1,0)^2,(0,1)^2\bigr)$. Therefore, we have two logarithmic components in the space of foliations on $X$ with normal sheaf equal to $\mathcal{O}_X(2)$ (see \cite[Lemma 2.5]{LPT13}).

These components also have the following description:

\begin{enumerate}
    \item Consider the rational map 
    \[
    \Psi\colon \mathbb{P} H^0(X, \mathcal{O}_X(1)) \times \mathbb{P} H^0(X, \mathcal{O}_X(1)) \dashrightarrow \mathbb{P} H^0(X, \Omega_X^1 \otimes \mathcal{O}_X(2))
    \]
    \[
    ([f], [g]) \longmapsto [\omega = f\,dg - g\,df].
    \]
    Let $\mathcal{U}$ be the open subset of the projective space $\mathbb{P} H^0(X, \Omega_X^1 \otimes \mathcal{O}_X(2))$ defined by the saturated $1$-forms, that is, 
    \[
    \mathcal{U} = \{ [\omega] \in \mathbb{P} H^0(X, \Omega_X^1 \otimes \mathcal{O}_X(2)) ; \; \operatorname{cod} \operatorname{Sing}(\omega) \ge 2\}.
    \]
    We define the component $\operatorname{Log}(1,1)$ by the intersection of $\mathcal{U}$ with the closure of the image of $\Psi$:
    \[
    \operatorname{Log}(1,1) = \overline{\operatorname{Image}(\Psi)} \cap \mathcal{U}.
    \]
    
    \item Analogously, consider the rational map 
    \[
    \Phi\colon \mathbb{P}(\mathbb{C}^2) \times \mathbb{P} H^0(X, \mathcal{O}_X(h_1))^{\times 2} \times \mathbb{P} H^0(X, \mathcal{O}_X(h_2))^{\times 2} \dashrightarrow \mathbb{P} H^0(X, \Omega_X^1 \otimes \mathcal{O}_X(2))
    \]
    \[
    (([\lambda_1:\lambda_2]), [f_1], [g_1], [f_2], [g_2]) \longmapsto [\omega = \lambda_1 f_2 g_2 (f_1\,dg_1 - g_1\,df_1) + \lambda_2 f_1 g_1 (f_2\,dg_2 - g_2\,df_2)].
    \]
    Consequently, the component $\operatorname{Log}\bigl((1,0)^2,(0,1)^2\bigr)$ is defined by the intersection 
    \[
    \operatorname{Log}\bigl((1,0)^2,(0,1)^2\bigr) = \overline{\operatorname{Image}(\Phi)} \cap \mathcal{U}.
    \]
\end{enumerate}
Now, suppose that $n \ge 3$. By Proposition~\ref{Picard 2-Structural}, either the foliation $\mathcal{F}$ is algebraically integrable, or $\mathcal{F}$ is defined by a closed rational $1$-form without divisorial components in its zero set. In the first case, by Proposition~\ref{algint-picard2}, the foliation $\mathcal{F}$ is given by a pencil of hyperplane sections and hence belongs to the component $\operatorname{Log}(1,1)$. In the latter case, since every such foliation lies in the closure of a logarithmic $1$-form (see \cite[Section 3]{Pereira22}), this leads to the conclusion that such a foliation belongs to one of the two logarithmic components described above. Consequently, the irreducible components of the space of codimension-one foliations on $X$ with normal sheaf equal to $\mathcal{N} = \mathcal{O}_X(2)$ are $\operatorname{Log}(1,1)$ and $\operatorname{Log}\bigl((1,0)^2,(0,1)^2\bigr)$.
\end{proof}

\subsubsection{Foliations on the adjoint variety of \(\mathfrak{sl}(3)\)}\label{S: adjoint sl3}

In this section, we consider $X$ to be the adjoint variety associated with the Lie algebra $\mathfrak{sl}(3)$, which can be realized as a smooth hyperplane section of $\mathbb{P}^2 \times \mathbb{P}^2$. Since we focus on codimension-one foliations with normal sheaf equal to $\mathcal{O}_X(2)$ and the canonical bundle of $X$ is given by $\omega_X = \mathcal{O}_X(-2)$, such foliations have trivial canonical bundle.
According to Proposition~\ref{Picard 2-Structural}, either these foliations are logarithmic, and their description is given in Proposition~\ref{P: logarithmic on PnxPn}, or these foliations are algebraically integrable. 

To construct new examples of algebraically integrable foliations on $X$, we investigate the action on $X$ of Lie subgroups $G \subset \operatorname{Aut}(X) = \operatorname{SL}(3)$ (see Section~\ref{E: fol group actions}). The two-dimensional orbits of such an action induce a foliation on $X$ with tangent sheaf isomorphic to  
\[
T_{\mathcal{F}} \simeq \mathfrak{g} \otimes \mathcal{O}_X \cong \mathcal{O}_X^{\oplus 2},
\]
where $\mathfrak{g}$ is the Lie algebra of $G$.
We begin by studying the foliation on $X$ associated with the affine Lie algebra $\mathfrak{aff}(\mathbb{C}) \subset \mathfrak{sl}(3) = H^0(X, T_X)$. In the subsequent results, we will show that this is the only instance in which a foliation arises from a new irreducible component in the space of foliations with trivial canonical bundle.

\begin{example}\label{example-aff}
Consider the action of the affine subgroup $\operatorname{Aff}(\mathbb{C}) = \mathbb{C}^* \ltimes \mathbb{C} \subset \operatorname{Aut}(\mathbb{P}^1)$ on $\mathbb{P}^2$ induced by the second Veronese embedding $\nu_2 \colon \mathbb{P}^1 \hookrightarrow \mathbb{P}^2$. Under this action, there is a unique closed orbit, given by $X_2 = \nu_2(\mathbb{P}^1)$, and infinitely many open $\operatorname{Aff}(\mathbb{C})$-orbits corresponding to binary forms with simple roots.  
The affine subgroup also acts on the product $\mathbb{P}^2 \times \mathbb{P}^2$, preserving the hyperplane defined by the vanishing trace condition. Consequently, the group $\operatorname{Aff}(\mathbb{C})$ acts on $X$, inducing a codimension-one foliation $\mathcal{F}$ on $X$ with trivial canonical bundle.

Moreover, $H^0(X, T_{\mathcal{F}})$ is isomorphic to the affine Lie algebra $\mathfrak{aff}(\mathbb{C}) = \langle x, y \rangle$, with the relation $[x, y] = y$.  
The foliation $\mathcal{F}$ leaves invariant two distinct surfaces in the linear systems $|\mathcal{O}_X(2 h_1)|$ and $|\mathcal{O}_X(2 h_2)|$. These surfaces correspond to the intersections of the products $\nu_2(\mathbb{P}^1) \times \mathbb{P}^2$ and $\mathbb{P}^2 \times \nu_2(\mathbb{P}^1)$ with the hyperplane defining $X$. Therefore, $\mathcal{F}$ does not belong to the component $\operatorname{Log}(1,1)$. 

Additionally, any algebraically integrable foliation in the component $\operatorname{Log}\bigl((1,0)^2,(0,1)^2\bigr)$ belongs to the intersection of the logarithmic components, implying that $\mathcal{F}$ is not contained in $\operatorname{Log}\bigl((1,0)^2,(0,1)^2\bigr)$.  
By computing $d\omega$, we verify that $\mathcal{F}$ has no non-Kupka singularities. Then, according to \cite[Theorem A]{PS24} (see also \cite[Theorem 51]{velazquez2024moduli}), we conclude that this foliation is rigid.  
Denote by $\operatorname{Aff}(\mathbb{C})$ the irreducible component of the space $\operatorname{Fol}(X, \mathcal{O}_X(2))$ whose general element is conjugate to the foliation $\mathcal{F}$ associated with the $2$-dimensional affine Lie algebra.
\end{example}

\begin{theorem}\label{teo-comp-A2}
Let $X$ be a smooth hyperplane section of $\mathbb{P}^2 \times \mathbb{P}^2$.
The irreducible components of the space of codimension-one foliations on $X$ with trivial canonical bundle are $\operatorname{Aff}(\mathbb{C})$, $\operatorname{Log}(1,1)$, and $\operatorname{Log}\bigl((1,0)^2,(0,1)^2\bigr)$.
\end{theorem}

Theorem \ref{teo-comp-A2} is a direct consequence of the following lemma and the three next propositions.

\begin{lemma}[{\cite[Lemma 4.2]{LPT13}}]
Let $X$ be a projective $3$-fold, $\mathcal{G}$ a one-dimensional foliation on $X$ with isolated singularities, and $\mathcal{F}$ a codimension-one foliation containing $\mathcal{G}$. If $H^1(X, K_X \otimes K_{\mathcal{G}}^{\otimes 2} \otimes N_{\mathcal{F}}) = 0$, then $T_{\mathcal{F}} \cong T_{\mathcal{G}} \oplus T_{\mathcal{H}}$ for a suitable one-dimensional foliation $\mathcal{H}$ on $X$.
\end{lemma}

\begin{proposition}\label{C*-action-prop}
Let $X$ be a smooth hyperplane section of $\mathbb{P}^2 \times \mathbb{P}^2$.
Let $\mathcal{F}$ be a codimension-one foliation on $X$ with trivial canonical bundle. Suppose that there exists an algebraic $\mathbb{C}^*$-action with non-isolated fixed points tangent to $\mathcal{F}$. Then $\mathcal{F}$ is defined by a closed rational $1$-form without divisorial components in its zero set.
\end{proposition}

\begin{proof}
We start by fixing the coordinates $((x_0:x_1:x_2), (y_0:y_1:y_2))$ on $\mathbb{P}^2 \times \mathbb{P}^2$. Assume that $X \subset \mathbb{P}^2 \times \mathbb{P}^2$ is given by the equation $\{x_0 y_0 + x_1 y_1 + x_2 y_2 = 0\}$ and that a $\mathbb{C}^*$-action on $X$ is given by
\[
\phi_{(a,b)} \colon \mathbb{C}^* \times X \longrightarrow X,
\]
\[
(\lambda, (x_0:x_1:x_2), (y_0:y_1:y_2)) \longmapsto (x_0: x_1 \lambda^a: x_2 \lambda^b, \; y_0: y_1 \lambda^{-a}: y_2 \lambda^{-b}),
\]
with $a, b \in \mathbb{N}$ relatively prime, since $\operatorname{Aut}(X) = \operatorname{SL}(3)$ has rank $2$. If $a$ and $b$ are distinct non-zero natural numbers, then the fixed points of the action are isolated. Hence, we have to analyze two cases:
\begin{enumerate}
    \item $(a, b) = (0, 1)$: Consider the rational map
    \[
    \phi \colon \mathbb{P}^2 \times \mathbb{P}^2 \dashrightarrow \mathbb{P}^1 \times \mathbb{P}^1,
    \]
    \[
    ((x_0:x_1:x_2), (y_0:y_1:y_2)) \longmapsto (x_0:x_1, y_0:y_1).
    \]
    Denote by $\phi_0$ the restriction of $\phi$ to $X$. The general fiber of $\phi$ corresponds to an orbit of the action induced by $\phi_{(0,1)}$. Hence, the foliation $\mathcal{F}$ is the pullback of a foliation $\mathcal{G}$ on $\mathbb{P}^1 \times \mathbb{P}^1$. Since $\phi_0^* \mathcal{O}_{\mathbb{P}^1 \times \mathbb{P}^1}(1) = \mathcal{O}_X(1)$ and the critical set of $\phi$ has codimension $2$, it follows that $N_{\mathcal{G}} = \mathcal{O}_{\mathbb{P}^1 \times \mathbb{P}^1}(2)$. Consequently, $\mathcal{G}$ is given by a closed rational $1$-form. Thus, the foliation $\mathcal{F}$, the pullback of $\mathcal{G}$ by $\phi_0$, is defined by a closed rational $1$-form without divisorial components in its zero set.

    \item $(a, b) = (1, 1)$: Consider the rational map
    \[
    \phi \colon \mathbb{P}^2 \times \mathbb{P}^2 \dashrightarrow \mathbb{P}^4,
    \]
    \[
    ((x_0:x_1:x_2), (y_0:y_1:y_2)) \longmapsto (x_0 y_0 : x_1 y_1 : x_1 y_2 : x_2 y_1 : x_2 y_2).
    \]
    Fix the coordinates $(z_0 : z_1 : z_2 : z_3 : z_4)$ in $\mathbb{P}^4$. Then $X$ is mapped into the intersection of the hyperplane $H = \{z_0 + z_1 + z_4 = 0\}$ and the quadric $Q = \{z_2 z_3 - z_1 z_4 = 0\}$. We denote by $\phi_0 \colon X \dashrightarrow \mathbb{P}^1 \times \mathbb{P}^1$ the restriction of $\phi$ to $X$.

    Since the monomials defining $\phi$ are invariant under the action, the dimension of the general fiber of $\phi_0$ is one. This implies that the general fiber corresponds to an orbit of the action induced by $\phi_{(1,1)}$. Hence, the foliation $\mathcal{F}$ is the pullback of a foliation $\mathcal{G}$ on $\mathbb{P}^1 \times \mathbb{P}^1$. Notice that $\phi_0^* \mathcal{O}_{\mathbb{P}^1 \times \mathbb{P}^1}(1) = \mathcal{O}_X(1)$. The result follows similarly to the first case.
\end{enumerate}
\end{proof}

\begin{proposition}\label{C-action-prop}
Let $X$ be a smooth hyperplane section of $\mathbb{P}^2 \times \mathbb{P}^2$.
Let $\mathcal{F}$ be a codimension-one foliation on $X$ with trivial canonical bundle. Suppose that there exists an algebraic $\mathbb{C}$-action with non-isolated fixed points tangent to $\mathcal{F}$. Then $\mathcal{F}$ is defined by a closed rational $1$-form without divisorial components in its zero set.
\end{proposition}

\begin{proof}
Fix the coordinates $((x_0:x_1:x_2), (y_0:y_1:y_2))$ on $\mathbb{P}^2 \times \mathbb{P}^2$. Assume that $X \subset \mathbb{P}^2 \times \mathbb{P}^2$ is given by the equation $\{x_0 y_0 + x_1 y_1 + x_2 y_2 = 0\}$, and let $\phi \colon \mathbb{C} \times X \longrightarrow X$ be an algebraic $\mathbb{C}$-action. This action is of the form $\phi(t) = \exp(t \cdot n)$, where $n$ is a nilpotent element of the Lie algebra $\mathfrak{aut}(X) = \mathfrak{sl}(3)$. Since there are only two conjugacy classes of non-zero nilpotent elements in $\mathfrak{sl}(3)$, we have the following cases to analyze:
\[
n_1 = \begin{pmatrix}
    0 & 0 & 1 \\
    0 & 0 & 0 \\
    0 & 0 & 0
\end{pmatrix} \qquad \text{and} \qquad n_2 = \begin{pmatrix}
    0 & 1 & 0 \\
    0 & 0 & 1 \\
    0 & 0 & 0
\end{pmatrix}.
\]
Since the fixed points of the action associated with $n_2$ are isolated, we only need to consider the nilpotent element $n_1$. The action on $X$ is given by 
\[
(\lambda, (x_0:x_1:x_2), (y_0:y_1:y_2)) \mapsto (x_0 + \lambda x_2, x_1, x_2), (y_0, y_1, -\lambda y_0 + y_2).
\]

Consider the rational map 
\[
\psi \colon \mathbb{P}^2 \times \mathbb{P}^2 \dashrightarrow \mathbb{P}^4
\]
given by 
\[
((x_0:x_1:x_2), (y_0:y_1:y_2)) \longmapsto (x_0 y_0 + x_2 y_2 : x_1 y_0 : x_1 y_1 : x_2 y_0 : x_2 y_1).
\]
Fixing the coordinates $(z_0 : z_1 : z_2 : z_3 : z_4)$ in $\mathbb{P}^4$, the variety $X$ is mapped into the intersection of the hyperplane $H = \{z_0 + z_2 = 0\}$ and the quadric $Q = \{z_2 z_3 - z_1 z_4 = 0\}$. 
We denote by 
\(
\psi_0 \colon X \dashrightarrow \mathbb{P}^1 \times \mathbb{P}^1
\)
the restriction of $\psi$ to $X$. The general fiber of $\psi_0$ coincides with an orbit of the action associated with $n_1$. Note that the critical set of $\psi_0$ has codimension greater than two. Consequently, the foliation $\mathcal{F}$ is the pullback of a foliation $\mathcal{G}$ on $\mathbb{P}^1 \times \mathbb{P}^1$ with normal sheaf equal to $\mathcal{O}_{\mathbb{P}^1 \times \mathbb{P}^1}(2)$ and given by a closed rational $1$-form. Therefore, $\mathcal{F}$ is defined by a closed rational $1$-form without divisorial components in its zero set.
\end{proof}

\begin{proposition}
Let $X$ be a smooth hyperplane section of $\mathbb{P}^2 \times \mathbb{P}^2$.
Let $\mathcal{F}$ be a codimension-one foliation on $X$ with trivial canonical bundle. Suppose that $\mathcal{F}$ is induced by an algebraic action of a two-dimensional Lie subgroup of $\operatorname{Aut}(X) = \operatorname{SL}(3)$. Then $\mathcal{F}$ is defined by a closed $1$-form without codimension-one zeros, or $\mathcal{F}$ is conjugate to the foliation associated to the affine Lie algebra.
\end{proposition}

\begin{proof}
Let $G \subset \operatorname{Aut}(X)$ be the subgroup defining $\mathcal{F}$, and let $\mathfrak{g} \subset \mathfrak{sl}(3)$ be its Lie algebra. If $G$ is abelian, then we have three cases to analyze: $\mathbb{C}^* \times \mathbb{C}^*$, $\mathbb{C} \times \mathbb{C}^*$, or $\mathbb{C} \times \mathbb{C}$.  
In the first case, every element in $\mathfrak{g}$ is semisimple. Since the rank of $\mathfrak{sl}(3)$ is two, $\mathfrak{g}$ is a Cartan subalgebra of $\mathfrak{sl}(3)$. Hence, the action of $G$ on $X$ is given by a subgroup of the form  
\[
\varphi_{\lambda,\mu}((x_0:x_1:x_2), (y_0:y_1:y_2)) = (x_0: x_1\lambda^a: x_2 \mu^b, \; y_0: y_1 \lambda^{-a}: y_2 \mu^{-b}), \quad \lambda, \mu \in \mathbb{C}^*.
\]  
Therefore, we can find $\mathbb{C}^* \subset G$ inducing an algebraic action with non-isolated fixed points tangent to $\mathcal{F}$. Hence, we can apply Proposition \ref{C*-action-prop} to conclude that $\mathcal{F}$ is induced by a closed rational $1$-form without divisorial components in its zero set.  

In the last two cases, $\mathfrak{g}$ contains a nilpotent element $n$, which defines a $\mathbb{C}$-action on $X$, corresponding to an inclusion $\mathbb{C} \subset G$. If the corresponding action has non-isolated fixed points, Proposition \ref{C-action-prop} implies that the foliation $\mathcal{F}$ is defined by a closed rational $1$-form without divisorial components in its zero set.  
If the corresponding action has only isolated fixed points, the action corresponds to a nilpotent element of the form  
\[
n = \begin{pmatrix}
    0 & \lambda & 0 \\
    0 & 0 & \lambda \\
    0 & 0 & 0
\end{pmatrix}.
\]  
The centralizer $C(n)$ of $n$ in $\mathfrak{sl}(3)$ is given by matrices of the form  
\(
\begin{pmatrix}
    0 & a & b \\
    0 & 0 & a \\
    0 & 0 & 0
\end{pmatrix}.
\)  
Since $\mathfrak{g}$ is abelian, we have $\mathfrak{g} \subset C(n)$, and consequently $\mathfrak{g}$ contains another nilpotent element, which defines an algebraic action with non-isolated fixed points. Thus, by Proposition \ref{C-action-prop}, the result follows.  

On the other hand, if $G$ is not abelian, then its Lie algebra $\mathfrak{g}$ is isomorphic to the affine Lie algebra $\mathbb{C} x \oplus \mathbb{C} y$ with the relation $[x, y] = y$. Note that $y$ is a nilpotent element, and we have an algebraic action on $X$ induced by an inclusion $\mathbb{C} \subset G$. Analogously to the last cases, if this action has non-isolated fixed points, then applying Proposition \ref{C-action-prop}, the result follows. Suppose now that this action has isolated fixed points. Then $y$ must be of the form  
\[
\begin{pmatrix}
    0 & \lambda & 0 \\
    0 & 0 & \lambda \\
    0 & 0 & 0
\end{pmatrix},
\]  
and the elements $x$ in $\mathfrak{sl}(3)$ satisfying $[x, y] = y$ are of the form  
\(
\begin{pmatrix}
    a & 0 & 0 \\
    0 & 0 & 0 \\
    0 & 0 & -a
\end{pmatrix}.
\)  
Hence, up to automorphisms of $X$, there is only one foliation on $X$ which is not invariant by an algebraic action of a one-dimensional Lie subgroup with non-isolated fixed points: the foliation associated with the $2$-dimensional affine Lie algebra.
\end{proof}
\subsection*{Acknowledgements}
I would like to thank my PhD advisors  Carolina Araujo and Daniele Faenzi,  for all the support and guidance in the process of construction
 of this paper. I am also grateful to
Jorge Vitório Pereira, for all the discussions we had during his visit to Dijon in May 2024, which was crucial for the development of this work. I am  really grateful for all the discussions I had with  
Alan Muniz and Vladimiro Benedetti, I've learned a lot from them.   I also thank Calum Spicer  for his careful reading and for providing corrections and suggestions that helped improve the paper. Finally, I would  like to thank Gabriel Fazoli and Caio Melo for always being open to answering my questions about foliations.

I was supported by CNPq Grant Number 140986/2021-9 and  CAPES - PDSE - Programa de Doutorado Sanduíche no Exterior Grant Number 88881.846472/2023-01.
This work also was partially supported by the CAPES-COFECUB programme (project number: Ma 1017/24), funded by the French Ministry for Europe and Foreign Affairs, the French Ministry for Higher Education, and CAPES. 
 
 \bibliographystyle{alphaurl}
\bibliography{bib}
 
\end{document}